\UseRawInputEncoding
\documentclass[12pt,a4paper,reqno]{amsart}
\usepackage{cite}
\usepackage[margin=2.4cm]{geometry}
\usepackage{graphics, amsmath, enumitem, comment}
\usepackage{graphicx}
\usepackage{epsfig}

\newif\ifcolorcomments

\colorcommentstrue
\usepackage{amssymb}
\usepackage{amsmath,amsthm}
\usepackage{mathrsfs}
\usepackage{bbm}
\def\bc{\begin{center}}
\def\ec{\end{center}}
\def\be{\begin{equation}}
\def\ee{\end{equation}}

\def\N{\mathbb N}
\def\Z{\mathbb Z}

\def\R{\mathbb R}

\def\LL{\mathcal L}
\DeclareMathOperator{\diam}{diam}
\def\a{\alpha}

\def\Z{\mathbb Z}

\newtheorem{lem}{Lemma}[section]

\newtheorem{definition}{Definition}[section]
\newtheorem{theorem}{Theorem}[section]

\newtheorem{lemma}[lem]{Lemma}

\newtheorem{corollary}{Corollary}[section]

\theoremstyle{remark}
\newtheorem{remark}{\bf Remark}[section]

\numberwithin{equation}{section}
\usepackage{colortbl}
\newtheorem{que}{\bf {Question}}[section]

\newif\ifdraft\drafttrue
\makeatletter
\@namedef{subjclassname@2020}{2020 Mathematics Subject Classification}
\makeatother

\begin{document}


\title{On the Intersection of Dynamical Covering Sets with Fractals}

\author[Zhang-nan Hu]{Zhang-nan Hu}
\address{Zhang-nan Hu, School of Mathematics, South China University of Technology, Guangzhou, 510641, China}
\email{hnlgdxhzn@163.com}
\author{Bing Li*}\thanks{* Corresponding author}
\address{Bing Li, School of Mathematics, South China University of Technology, Guangzhou, 510641, China}
\email{scbingli@scut.edu.cn}
\author[Yimin Xiao]{Yimin Xiao}
\address{Yimin Xiao, Department of Statistics and Probability, Michigan State University, East Lansing, MI, 48824, USA}
\email{xiaoy@msu.edu}

\begin{abstract} 
Let $(X,\mathscr{B}, \mu,T,d)$ be a measure-preserving dynamical system with exponentially mixing property, and let $\mu$ 
be an Ahlfors $s$-regular probability measure. The dynamical covering problem concerns the set $E(x)$ of points which are covered 
by the orbits of $x\in X$ infinitely many times. We prove that the Hausdorff dimension of the intersection of $E(x)$ and any regular 
fractal $G$ with $\dim_{\rm H}G>s-\alpha$ equals $\dim_{\rm H}G+\alpha-s$, where $\alpha=\dim_{\rm H}E(x)$ $\mu$--a.e. Moreover, we obtain the packing 
dimension of $E(x)\cap G$ and an estimate for $\dim_{\rm H}(E(x)\cap G)$ for any analytic set $G$. 
\end{abstract}

\maketitle

\section{Introduction}

Let $(X,d)$ be a compact metric space and let $(X,\mathscr{B},T,\mu,d)$ be a metric measure preserving system (m.m.p.s. for short). 
The distribution of the orbit of a point in $X$ is an important topic in ergodic theory and has been studied by many authors. See, 
for example,  \cite{AP,Bo,BW,FLL,FanScT13,fur,HLSW}. The well-known Poincar\'e recurrence theorem shows that $\mu$-a.e. $x\in X$ is 
recurrent, that is
\[\liminf_{n\to\infty}d(T^nx,x)=0.\]
Boshernitzan \cite{Bo} proved that if there is some $\tau>0$ such that the $\tau$-dimensional Hausdorff measure $\mathcal{H}^{\tau}$ 
of $X$ is $\sigma$-finite, then for $\mu$-a.e. $x\in X$,
\[\liminf_{n\to\infty}n^{\frac{1}{\tau}}d(T^nx,x)<\infty.\]
If $\mu$ is ergodic, then for every fixed point $y\in X$, we have, for $\mu$-a.e. $x\in X$,
\[\liminf_{n\to\infty}d(T^nx,y)=0.\]
For an exponentially mixing metric measure preserving system, Fan, Langlet and Li \cite{FLL} proved that if $t<1/\alpha_{max}$, 
then for $\mu$-a.e. $x\in X$, we have $\liminf\limits_{n\to\infty}n^td(T^nx,y)=0$ for uniformly for all $y\in X$, where $\alpha_{max}$ 
is the maximal local dimension of $\mu$.

Hill and Velani \cite{HV} introduced the shrinking targets theory, which concerns the following set of points whose orbits are close to 
a given point, that is for any given $y \in X$, 
\begin{equation}\label{Eq:HV}
S(y) = \{x\in X\colon T^nx\in B(y,\ell_n)~ \, {\rm i.o.}\},
\end{equation}
where $\{\ell_n\}_{n\ge1}$ is a sequence of positive real numbers tending to 0 and i.o. stands for infinitely often. Li, Wang, 
Wu and Xu \cite{LWWX} studied the shrinking 
target problem in the case when $T$ is the Gauss map and determined the Hausdorff dimension of the set $S(y)$ in 
(\ref{Eq:HV}) for certain choices of $\{\ell_n\}$.  Bugeaud and Wang \cite{BW} studied the problem for the case when $T$ 
is the $\beta$-transformation. Aspenberg and Persson \cite{AP} extended their results to piecewise expanding maps.

Motivated by the Diophantine approximation, Fan, Schmeling and Troubetzkoy~\cite{FanScT13} proposed the dynamical 
covering set defined by
\begin{equation}\label{Eq:DyCov}
E(x)=\{y\in X\colon T^nx\in B(y,\ell_n) ~{\rm i.o. }\},
\end{equation} 
which is the set of points $y$ that are well approximated by the orbit of $x$. Among other interesting results, they considered the 
case when $X$ is the unit interval and $T\colon x\mapsto 2x$ $(\bmod\thinspace1)$ and computed $\dim_{\rm H} E(x)$, where 
$\dim_{\rm H}$ denotes the Hausdorff dimension. In \cite{Liao}, Liao and Seuret determined $\dim_{\rm H}E(x)$ when $T$ is an 
expanding Markov map. Later, Persson and Rams \cite{PR17} considered more general piecewise expanding maps than the 
Markov maps.
 
In 2017, Wang, Wu and Xu \cite{Wang} considered the case when $X$ is the middle-third Cantor set, $Tx=3x$ (mod 1), and $\mu$ 
is the standard Cantor measure. They gave a complete characterization of the size $E(x)$ for $\mu$-almost all $x$. In \cite{ZL}, 
Hu and Li investigated the dynamical covering sets in $(X,\mathscr{B},T,\mu,d)$ with exponentially mixing property, where $\mu$ is 
an Ahlfors $s$-regular Borel probability measure. They showed that the measure $\mu(E(x))$ is 0 or 1 for $\mu$-a.e. $x$ according 
to the convergence or divergence of the series $\sum_{n=1}^{\infty}\ell_n^s$, and for $\mu$-a.e. $x$,
\[\dim_{\rm H}E(x) = \a,\]
where $\a$ is the {\it upper Besicovitch-Taylor index} of $\{\ell_n\}_{n\ge1}$ defined by
\begin{equation} \label{Def:BTindex}
\a:=\inf\bigg\{t\le s\colon\sum_{n=1}^{\infty}\ell_n^t<\infty\bigg\}=\sup\bigg\{t\le s\colon\sum_{n=1}^{\infty}\ell_n^t=\infty\bigg\}.
\end{equation}
 In particular, these results hold when $X$ is the middle-third Cantor set and $\mu$ is the standard Cantor measure.

Motivated by the aforementioned research, we are interested in following natural questions for the dynamical covering set.
\begin{que}\label{main1}
 For a given set $G\subset X$, are there points in $G$ which can be well approximated by the orbit of $x \in X$?
 Equivalently, when is $E(x)\cap G\ne\varnothing$ for $x\in X$? 
\end{que}
\begin{que}\label{main2}
	If $E(x)\cap G\ne\varnothing$, then how large is the intersection?
\end{que}
 
Question \ref{main1} is also closely related to Mahler's question \cite{Mahler} which is concerned with approximating the 
numbers in the middle-third Cantor set $C_{1/3}$ by rational numbers and bears some analogy with the dynamical covering 
set. Several authors have investigated Mahler's question by measuring the size of the intersection set 
$\mathcal{W}_{\mathcal{A}}(\psi)\cap C_{1/3}$, where $\mathcal{A}$ is an infinite subset of $\mathbb{N}$ and
\[\mathcal{W}_{\mathcal{A}}(\psi)=\bigg\{x\in \R\colon \Big|x-\frac{p}{q}\Big|\le \psi(q) ~{\rm for ~infinitely~many}~(p,q)
\in\mathbb{Z}\times \mathcal{A}\bigg\},
\]
and $\psi$ is the approximation speed. For example, when $\mathcal{A}:=\{3^n\colon n=0,1,2\dots\}$, Levesley, Salp and 
Velani \cite{LSV} studied the $f$-Hausdorff measure $\mathcal{H}^f$ of the set $\mathcal{W}_{\mathcal{A}}(\psi)\cap C_{1/3}$, 
where $f$ is a measure function, and provided a criterion for $\mathcal{H}^f(\mathcal{W}_{\mathcal{A}}(\psi)\cap C_{1/3})$ 
to be 0 or $\mathcal{H}^f(C_{1/3})$. As for the case when ${\mathcal{A}}=\N$, the problem is still open.  
Levesley, Salp and Velani \cite{LSV} conjectured that if $\psi (q) = q^{-\tau}$ with $ \tau\ge2$ then
\begin{equation}\label{conj1}
\dim_{\rm H}(\mathcal{W}_{\N}(q^{-\tau})\cap C_{1/3})=\frac{2}{\tau}\dim_{\rm H} C_{1/3}. 
\end{equation}
However, Bugeaud and Durand \cite{YA} disagreed with (\ref{conj1}) and proposed another conjecture:
\begin{equation}\label{conj}
\dim_{\rm H}(\mathcal{W}_{\N}(q^{-\tau})\cap C_{1/3})= \max\bigg\{\dim_{\rm H}(\mathcal{W}_{\N}(q^{-\tau}))
+\dim_{\rm H} C_{1/3}-1,\, \frac{1}{\tau}\dim_{\rm H} C_{1/3}\bigg\}.
\end{equation}
Bugeaud and Durand \cite{YA} provided some results to support their conjecture. In particular, they showed that 
(\ref{conj}) holds for a natural probabilistic model which is a random covering set on the unit circle $\mathbb{T}$ 
generated by intervals whose centers are uniformly distributed independent random variables in $\mathbb{T}$
(see \cite[Section~2]{YA}). Recently, Yu \cite{Hanyu} proved that the conjecture (\ref{conj}) holds for the middle-$p$th 
Cantor set when $p>10^7$ is odd and $\tau\in (1,1+c)$ for some number $c>0$. Here, for each odd integer $p>2$, 
the middle-$p$th Cantor set is the set of numbers whose base $p$ expansions do not have digit $(p-1)/2$.

Before stating the main results of this paper, we recall some definitions that will be used throughout this paper.
  
\begin{definition}
A Borel measure $\mu$ on $(X,\mathscr{B})$ is called Ahlfors~s-regular $(0< s < \infty)$ if there exists a constant 
$1\le c_1 <\infty$ such that 
\begin{equation}\label{eqah}
	{c_1}^{-1}r^s\le\mu\big(B(x,r)\big)\le c_1\,r^s
\end{equation} 
for all $x\in X$ and $0<r\le\diam X$, where $B(x,r)$ is the closed ball in metric $d$ whose center is $x$ 
with radius $r$ and $\diam X$ is the diameter of $X$. 
\end{definition}

\begin{definition}
A m.m.p.s. $(X,\mathscr{B},\mu,T,d)$ is $\textit{exponentially mixing}$ if there exist two constants $C>0$ and $0<\rho<1$ 
such that 
\[|\mu(E|T^{-n}F)-\mu(E)|\le C\rho^n\]
for all $n\ge1$, balls $E \subseteq X$, and measurable sets $F\in\mathscr{B}$ with $\mu(F)>0$. Here $\mu(A|B)$ denotes 
the conditional measure $\frac{\mu(A\cap B)}{\mu(B)}$. Sometimes we say $\mu$ is exponentially mixing.
\end{definition}

Throughout we 
denote packing dimension and upper box dimension in the metric space $(X,d)$ by  $\dim_{\rm P}$ and 
$\overline{\dim}_{\rm B}$, respectively. We adopt the convention that the Hausdorff dimension and packing dimension of empty sets are equal to $-\infty$ as in \cite{YA} to distinguish the empty set from a non-empty set with dimension 0.

Theorem \ref{dc3} provides a criterion for 
Question \ref{main1}. The condition (C) in Theorem \ref{dc3} is stated in the Section 2. 

\begin{theorem}\label{dc3}
Let $(X,\mathscr{B},\mu,T,d)$ be an exponentially mixing m.m.p.s. and the measure $\mu$ be Ahlfors s-regular 
with $0<s<\infty$. Let $\{\ell_n\}_{n\ge1}$ be a sequence of positive numbers tending to 0 with the upper 
Besicovitch-Taylor index $\alpha<s$ and, for any $x \in X$,  let $E(x)$ be the dynamical covering set defined in 
\eqref{Eq:DyCov}. Then  for any analytic set $G\subset X$ we have, for $\mu$-almost every $x\in X$
 \begin{align*}
E(x)\cap G \left\{\begin{array}{cl}
=\varnothing& \quad \text{if~}\dim_{\rm P}(G)<s-\alpha,\\[2ex]
\ne\varnothing&\quad \text{if~}\dim_{\rm H}(G)>s-\alpha.
 \end{array}\right.
\end{align*}
 If, in addition, the condition (C) holds, then 
\begin{equation*}
	E(x)\cap G\ne\varnothing  \quad \text{if~ $\dim_{\rm P}(G)>s-\alpha$}.
\end{equation*}
\end{theorem}

Our Theorem \ref{dc1} is concerned with Question \ref{main2} and measures the size of $E(x)\cap G$.

\begin{theorem}\label{dc1}
Let $E(x)$ be the dynamical covering set as in Theorem \ref{dc3}. For any analytic set $G\subset X$ we have  for $\mu$-almost every $x\in X$
\begin{align*}
            \dim_{\rm H}(E(x)\cap G) \left\{\begin{array}{ll}
            \le \dim_{\rm P}(G)+\alpha-s& \quad \text{if~}\dim_{\rm P}(G)\ge s-\alpha,\\[2ex]
             =-\infty & \quad \text{if~}\dim_{\rm P}(G)< s-\alpha,\\[2ex]
             \ge\dim_{\rm H}(G)+\alpha-s &\quad \text{if~}\dim_{\rm H}(G)>s-\alpha.
          \end{array}\right.
          \end{align*}
Moreover, if $\dim_{\rm P}(G)>s-\alpha$ and the condition (C) holds, then
\[\dim_{\rm P}(E(x)\cap G)=\dim_{\rm P}(G),\quad{\rm a.e.}\]
\end{theorem}

\begin{remark}
In Theorem \ref{dc3} and \ref{dc1}, the critical case where $\dim_{\rm P}(G)\ge s-\alpha$ and $\dim_{\rm H}(G)\le s-\alpha$ does not explicitly.
\end{remark}

As an immediate consequence of Theorem \ref{dc1}, we have

\begin{corollary}\label{dc2}
For any regular analytic set $G\subset X$ in the sense that 
$\dim_{\rm H}(G)=\dim_{\rm P}(G)$,  we have for $\mu$-almost every $x\in X$
\begin{align*}
           \dim_{\rm H}(E(x)\cap G) =\left\{\begin{array}{ll}
           \dim_{\rm P}(G)+\alpha-s& \quad \text{if~}\dim_{\rm P}(G)> s-\alpha,\\[2ex]
           -\infty & \quad \text{if~}\dim_{\rm P}(G)< s-\alpha.
           \end{array}\right.
\end{align*}
\end{corollary}
The rest of this article is organized as follows. In Section 2, we describe a more general probabilistic setting and prove results 
on the hitting probabilities of the random covering set. Then Theorems \ref{dc3}--\ref{dc1}  follow from Theorems \ref{hithaus1}--\ref{rc1}.                                                                                                                                                                                                                                                                                                                                                                                                                                                                                                                                                                                                                                                                                                             This probabilistic setting is closely related to the classical Dvoretzky covering problem concerning the set $\limsup\limits_{n\to\infty} 
B(x_n,\ell_n)$, where the centers $\{x_n\}$ are independent and uniformly distributed. J\"arvenp\"a\"a $et~al.$ \cite{JKLSX} studied 
the hitting probability of the Dvoretzky covering set in a general metric space. We extend their results from the independence setting 
to a stationary process which is exponentially mixing, and also the uniform distribution is generalized to Ahlfors regular distribution. 
These results will give the lower and upper bounds for $\dim_{\rm H}(E(x)\cap G)$ in Theorem \ref{hithaus2}. In Section 4, we derive 
Theorem \ref{rc1} by extending the general method of Khoshnevisan, Peres, and Xiao \cite{KPX} to limsup random fractals in metric 
spaces. Moreover, we obtain the packing dimension of $E(x)\cap G$ in Theorem \ref{rc1}. The proofs of our main results in Section 2 
are given in Section 3 and Section 5, respectively. Finally, in Section 6, we provide some examples of dynamical systems that satisfy 
our assumptions so that the main theorems in the paper are applicable to them.

\section{General results for stationary processes}

In this section, we investigate Questions \ref{main1} and \ref{main2} in a general probabilistic setting, which extends the results 
in \cite{YA} and \cite{JKLSX}.

Let $\{\xi_n\}_{n\ge1}$ be a stationary process defined on a probability space $(\Omega,\mathscr{A},\mathbb{P})$ 
and take values in a compact metric space $(X,d)$. Let $\mu$ be the the distribution of $\xi_1$ which is the probability 
measure defined by 
\begin{equation}\label{eqmu}
	\mu(A)=\mathbb{P}(\xi_1\in A)
\end{equation}
for all Borel sets $A \subset X$. 

\begin{definition}\label{exp}
We say that $\{\xi_n\}_{n\ge1}$ is $exponentially~ mixing$ if there exist two constants $C>0$ and $0<\rho<1$
such that  
\[\big|\mathbb{P}(\xi_1\in A|D)-\mathbb{P}(\xi_1\in A)\big|\le C\rho^{n}
\]
for all $n\ge1$, balls $A\subset X$ and $D\in\mathscr{A}^{n+1}$, where $\mathscr{A}^{n+1}$ is the sub-$\sigma$-field 
generated by $\{\xi_{n+i}\}_{i\ge1}$.
\end{definition}

\begin{remark}
If $(X,\mathscr{B},\mu,T,d)$ is an exponentially mixing m.m.p.s., define $\xi_n:=T^{n-1}x$ for every $x\in X$, then the 
process $\{\xi_n\}_{n\ge1}$ is an exponentially mixing process on probability space $(X,\mathscr{B},\mu)$. Hence the 
probabilistic model  described above covers the dynamical case.
\end{remark}

Let $\{\ell_n\}_{n\ge1}$ be a sequence of positive numbers decreasing to zero. For every $n\ge 1$, denote $I_n:= 
B(\xi_n,\ell_n)$. Define 
\[
E:=\limsup_{n\to\infty} I_n=\{y\in X\colon y\in I_n ~\text{i.o.}\}.
\]
The set $E$ is a $random ~covering~ set$ and consists of the points which are covered by $\{I_n\}_{n\ge1}$ infinitely 
often. 

The following theorem is concerned with the hitting probabilities of the random  covering set $E$.

\begin{theorem}\label{hithaus1}
Let $\{\xi_n\}_{n\ge1}$ be an exponentially mixing stationary process taking values in $X$ with probability distribution $\mu$. 
We assume that $\mu$ is Ahlfors $s$-regular with $0<s<\infty$. Let $\a$ be the upper Besicovitch-Taylor index of 
$\{\ell_n\}_{n\ge1}$ with $\alpha<s$. Then for every analytic set $G\subset X$, we have
	\[\mathbb{P}(E\cap G\ne\varnothing)=
	\begin{cases}
	0&  \text{if $\dim_{\rm P}(G)<s-\alpha$,}\\
	1&  \text{if $\dim_{\rm H}(G)>s-\alpha$.}
	\end{cases}\]
\end{theorem}

The theorem below provides an estimate on the Hausdorff dimension of the intersection $E\cap G$.
\begin{theorem}\label{hithaus2}
	Under the setting of Theorem \ref{hithaus1}, for every analytic set $G\subset X$, with probability one we have  
	 \begin{align*}
            \dim_{\rm H}(E\cap G) \left\{\begin{array}{ll}
            \le \dim_{\rm P}(G)+\alpha-s& \quad \text{if~}\dim_{\rm P}(G)\ge s-\alpha,\\[2ex]
             =-\infty & \quad \text{if~}\dim_{\rm P}(G)< s-\alpha,\\[2ex]
             \ge\dim_{\rm H}(G)+\alpha-s &\quad \text{if~}\dim_{\rm H}(G)>s-\alpha.
          \end{array}\right.
          \end{align*}
\end{theorem}

The following corollary is an immediate consequence of Theorem \ref{hithaus2}.
\begin{corollary}\label{hithaus3}
Under the setting of Theorem \ref{hithaus1}, for any analytic set $G\subset X$ with $\dim_{\rm H}(G)=\dim_{\rm P}(G)=\gamma$ we have
\begin{align*}
           \dim_{\rm H}(E\cap G) =\left\{\begin{array}{ll}
           \gamma+\alpha-s& \quad \text{if~}\gamma> s-\alpha,\\[2ex]
           -\infty & \quad \text{if~}\gamma< s-\alpha.
           \end{array}\right.
\end{align*}
\end{corollary}

In general, in Theorem \ref{hithaus1}, $\mathbb{P}(E\cap G\ne\varnothing)=1$ may not hold if $\dim_{\rm H}G>s-\a$ is 
replaced by the weaker condition $\dim_{\rm P}G>s-\a$. A counter-example was given by Li and Suomala \cite{LS} 
when $\{\xi_n\}$ is a sequence of independent and uniformly distributed random variables on the circle. Therefore in 
the result below, we will make use of the following condition on the sequence $\{\ell_n\}$:

(C)\ Let $b\in(0,\frac{1}{3})$ be a constant. There exists an increasing sequence of positive integers $\{k_i\}$ such that
$k_i \to \infty$ as $i \to \infty$, 
\begin{equation}\label{ki}
	\lim_{i\to\infty}\frac{k_{i+1}}{k_i}=1
\end{equation}
and
\begin{equation}\label{nki}
	\lim_{i\to\infty}\frac{\log_{b^{-1}}n_{k_i}}{k_i}=\alpha,
\end{equation}
where 
\[n_k=\#\{n\ge1\colon \ell_n\in[b^{k-1},b^{k-2})\}.\]
\begin{remark}
By Li, Shieh and Xiao \cite{LSX}, the upper Besicovitch-Taylor index $\a$ of $\{\ell_n\}_{n\ge 1}$ can be expressed as  
\begin{equation}\label{alimsup}
\alpha=\limsup_{k\to\infty}\frac{\log_{b^{-1}}n_k}{k}.
\end{equation}
\end{remark}

Under Condition (C), we are able to improve Theorem \ref{hithaus1} by showing that the hitting probability of the random 
covering set $E$ with an arbitrary analytic set $G \subseteq X$ is determined by the packing dimension of $G$. 
\begin{theorem}\label{rc1}
	Under the setting of Theorem \ref{hithaus1}, if the condition (C) holds, then for every analytic set $G\subset X$ with 
	$\dim_{\rm P}(G)>s-\alpha$, we have
	\[\mathbb{P}(E\cap G\ne\varnothing)=1.\]
Moreover, if $\dim_{\rm P}(G)>s-\alpha$, then $\dim_{\rm P}(E\cap G)=\dim_{\rm P}(G)$ {\rm a.s.}	
\end{theorem}

We end this section with some remarks about studies of random covering sets. The random covering problem goes back 
to 1897 when Borel investigated questions related to random placement of circular arcs in the unit circle \cite{Borel}. Let 
$\xi=\{\xi_n\}$ be a sequence of independent and uniformly distributed random variables on the circle $\mathbb{T}$ 
and $\{\ell_n\}$ be a sequence of positive numbers decreasing to 0, define the random covering set
$$E(\xi)=\limsup_{n\to\infty}B(\xi_n,\ell_n).$$
In 1956, Dvoretzky \cite{Dvor56} called the attention on the study of $E(\xi)$. He asked the question when $E(\xi)=\mathbb{T}$ 
a.s. or not. In 1971, Shepp~\cite{Shepp72} gave a sufficient and necessary condition: $E(\xi)=\mathbb{T}$ a.s. if and only if 
$\sum_{n=1}^{\infty}\frac{1}{n^2}\exp(\ell_1 +\dots+\ell_n)=\infty$. For the high dimensional case, the Dvoretzky covering problem 
for balls is still open (more details see \cite{Kah00})

The Hausdorff dimensions of random covering sets were firstly investigated by Fan and Wu \cite{FW}. J\"arvenp\"a\"a $et~al.$ \cite{JJKLS} 
introduced the self-affine covering sets and obtained the dimension formula in terms of the singular value function, which was generalized to any 
Lebesgue measurable set covering by Feng $et~al.$ \cite{FJS}.

The random covering problem is related to many other fields such as number theory. For example, in 2017, Haynes and Koivusalo \cite{HK17} 
used a covering argument in Dvoretzky \cite{Dvor56} to prove the randomized version of the Littlewood Conjecture.



\section{Proofs of Theorem \ref{hithaus1} and Theorem \ref{hithaus2} }

\subsection{A nesting family in a metric space}
\ 
\newline
\\
\indent 
We start this section by recalling Theorem 2.1 of K\"aenm\"aki, Rajala and Suomala \cite{Krs}, which provides a nesting family 
of ``cubes"  in a metric space with the finite doubling property [i.e., every ball $B(x, 2r) \subset X$ may be covered by finitely many balls of radius $r$.]  This family  
shares most of the good properties of dyadic cubes of Euclidean spaces. 

\begin{theorem}[\cite{Krs}]\label{thm:nest}
Let $(X,d)$ be a metric space with the finite doubling property and let $0<b<\frac{1}{3}$ be a constant. Then there exists a collection 
$\{Q_{k,i}\colon k\in\mathbb{Z}, i\in\mathbb{N}_k\subset \mathbb{N}\}$ of Borel sets that have the following properties:
\begin{enumerate}
	\item $X=\bigcup_{i\in\mathbb{N}_k}Q_{k,i}$ for every $k\in\mathbb{Z}$.
		\item $Q_{k,i}\cap Q_{m,j}=\varnothing $ or $Q_{k,i}\subset Q_{m,j}$, where $k,m\in \mathbb{Z}$, $k\ge m$, 
		$i\in \mathbb{N}_k $ and $j\in \mathbb{N}_m$.
		\item For every $k\in \mathbb{Z}$ and $i\in\mathbb{N}_k$, there exists a point $x_{k,i}\in X$ such that 
		\begin{equation}\label{inou}
		U(x_{k,i},c_2b^k)\subset Q_{k,i}\subset B(x_{k,i},c_2'b^k),
\end{equation}
where $c_2=\frac{1}{2}-\frac{b}{1-b}$, $c'_2=\frac{1}{1-b}$ and $U(x_{k,i},c_2b^k)$ is the open ball with center $x_{k,i}$ and 
radius $c_2b^k$. 
\item There exists a point $x_0\in X$ so that for every $k\in\mathbb{Z}$, there is an index $i\in\mathbb{N}_k$ with 
$U(x_0,c_2b^k)\subset Q_{k,i}$.
\item $\{x_{k,i}\colon i\in \mathbb{N}_k\}\subset \{x_{k+1,i}\colon i\in \mathbb{N}_{k+1}\} $ for all $k\in\mathbb{Z}$.
	\end{enumerate}
\end{theorem}

\begin{remark}\label{Qki}
From the construction of $\{Q_{k,i}\colon k\in\mathbb{Z}, i\in\mathbb{N}_k\subset \mathbb{N}\}$ in \cite{Krs} we see 
that for any $k\in\Z$, $Q_{k,i}\cap Q_{k,j}=\varnothing$ for $i\ne j\in\N_k$. When $\mu$ is Ahlfors $s$-regular, then by (1) and (3) in Theorem 
\ref{thm:nest} and the equalities (\ref{eqah}), we have 
\begin{equation}\label{Eq:count}
c_1^{-1}{c'}_2^{-s}b^{-ks}\le\#\mathbb{N}_k=\#\{Q_{k,i}\colon i\in \mathbb{N}_k\}\le c_1c_2^{-s}b^{-ks}.
\end{equation}
We will use this fact in our proofs of Theorems \ref{hithaus1}, \ref{rc1} and \ref{thm2} below.
\end{remark}

If $(X,d)$ is a compact metric space endowed with an Ahlfors $s$-regular measure $\mu$, then it can be verified that
$X$ has the finite doubling property. Let $b\in(0,\frac{1}{3})$ be the constant in the condition (C) and let 
$\{Q_{k,i}:k\in\mathbb{Z}, i\in\mathbb{N}_k\}$ be the nesting family as in Theorem \ref{thm:nest} which we call 
``generalized dyadic cubes" of $(X,d)$. 
For convenience, we write $\mathcal{Q}_0=\{X\}$ and $\mathcal{Q}_k=\{Q_{k,i}\colon i\in \mathbb{N}_k\}$ for $k\ge1$,  
 and $\mathcal{B}_k=\{B(x_{k,i},c'_2b^k)\colon  i\in \mathbb{N}_k\}$, where $x_{k,i} \ (i \in  \mathbb{N}_k)$ are the points 
 in Part (3) of Theorem \ref{thm:nest}. 


\begin{lem}\label{number}
Let $(X,d)$ be a compact metric space endowed with an Ahlfors $s$-regular measure $\mu$, where $0<s<\infty$ is a constant. 
Then, for any constants $a_0>0$ and $k\ge1$, a ball $B$ of radius ${a_0b^k}$ may intersect at most $\frac{c_1^2(2c_2'+a_0)^s}{c_2^s} $ 
elements in $\mathcal{Q}_k$, where $c_1, c_2$ and $c_2'$ are the constants given in \eqref{eqah} and Theorem \ref{thm:nest}, 
respectively.
\end{lem}
\begin{proof}
Write $B=B(x_B,a_0b^k)$ and
\[\mathcal{A}=\Big\{Q_{k,i}\in\mathcal{Q}_k\colon Q_{k,i}\cap B\ne\varnothing,~i\in\N_k\Big\}.\]
For $Q_{k,i}\in\mathcal{A}$, from Theorem \ref{thm:nest} (3), there exists one point $x_{k,i}\in Q_{k,i}$ so that (\ref{inou}) holds. 
 We denote the collection of such points by $\Upsilon$, and so $\#\mathcal{A}=\#\Upsilon$. Notice that 
\[B\subset \bigcup_{Q_{k,i}\in\mathcal{A}}Q_{k,i}\subset B(x_B,(2c'_2+a_0)b^k),\]
 and
 \[\bigcup_{x_{k,i}\in \Upsilon}U(x_{k,i},c_2b^k)\subset \bigcup_{Q_{k,i}\in\mathcal{A}}Q_{k,i}.\]
 Therefore
\[\sum_{x_{k,i}\in \Upsilon}\mu\Big(U(x_{k,i},c_2b^k)\Big)\le \mu\Big(\bigcup_{Q_{k,i}\in\mathcal{A}}Q_{k,i}\Big)\
le \mu\Big( B(x_B,(2c'_2+a_0)b^k)\Big).\]
Since $\mu$ is Ahlfors $s$-regular, we have
\[\sum_{x_{k,i}\in \Upsilon}\mu(U(x_{k,i},c_2b^k))\ge (\#\mathcal{A})(c_2b^k)^sc_1^{-1},\]
and
\[\mu( B(x_B,(2c'_2+a_0)b^k))\le c_1(2c'_2+a_0)^sb^{ks}.\]
Hence $\#\mathcal{A}\le \frac{c_1^2(2c'_2+a_0)^s}{c_2^s} $. This proves Lemma \ref{number}.
\end{proof}

\subsection{Proofs of Theorem \ref{hithaus1} and Theorem \ref{hithaus2}}
\ 
\newline
\\
\indent
For $k\ge1$, denote by $\mathfrak{I}_k =\{j\ge1\colon  \ell_j\in[b^{k-1},b^{k-2})\}$ and $n_k=\#\mathfrak{I}_k$. 
Given a constant $c>\frac{s-\alpha}{\log_{b} \rho}$, where $\rho$ is the constant in Definition \ref{exp}, 
let $\mathfrak{I}'_k$ be a maximal collection of $\mathfrak{I}_k$ having mutual distances at least $ck$. One important 
property of $\mathfrak{I}'_k$ is that any pair of integers $n, m \in \mathfrak{I}'_k $ are at least of distance $c k$ from 
each other. We will use this fact to prove Theorems \ref{hithaus1} and \ref{rc1}. See Remark \ref{pro} below for more details.

Write $m_k=\#\mathfrak{I}'_k$.  Since the sequence $\{\ell_n\}_{n\ge1}$ is decreasing, the elements in $\mathfrak{I}_k $ are 
consecutive to each other. Then $m_k=\lceil(c k)^{-1} n_k \rceil$, where $\lceil \cdot \rceil$ stands for the ceiling function.

\begin{lemma}\label{hauslem1}
Let $(X,d)$ be a compact metric space, and $G\subset X$ be an analytic set.\\
(1)~If $\dim_{\rm H}G>t$, there is a nonempty compact subset $G^{\star}\subset G$ such that
$$\dim_{\rm H}(G^{\star}\cap V)>t$$ for all open sets $V\subset X$ with $G^{\star}\cap V\ne\varnothing$.\\
(2)~If $\dim_{\rm P}G>t$, there is a nonempty compact subset $G_{\star}\subset G$ such that
 $$\dim_{\rm P}(G_{\star}\cap V)>t$$ for all open sets $V\subset X$ with $G_{\star}\cap V\ne\varnothing$.
\end{lemma}
The conclusions in Lemma \ref{hauslem1} are known. We include a proof for completeness.
\begin{proof}
Since $G$ is analytic, from \cite{How}, there exists a compact set $K\subset G$ with $0<\mathcal{H}^{t'}(K)<\infty$ for 
some $t'>t$. Let $\{V_i\}_{i\ge1}$ be a countable basis of $X$. Denote 
\begin{equation}\label{V}
\mathcal{V}=\{i\colon \dim_{\rm H}(K\cap V_i)\le t\},
\end{equation}
Then $\bigcup_{i\in\mathcal{V}}V_i$ is relatively open in $K$, and hence $G^{\star}=K\backslash\bigcup_{i\in\mathcal{V}}V_i$ 
is compact. Note that
\begin{equation*}
\begin{split}
\mathcal{H}^{t'}(K)&=\mathcal{H}^{t'}\bigg(K\cap \bigcup_{i\in\mathcal{V}}V_i \bigg) + \mathcal{H}^{t'} 
\bigg( K\backslash\bigcup_{i\in\mathcal{V}}V_i\bigg)\\
&\le \mathcal{H}^{t'}\bigg(K\cap \bigcup_{i\in\mathcal{V}}V_i\bigg)+ \mathcal{H}^{t'}(G^{\star})\\
&\le \sum_{i\in\mathcal{V}}\mathcal{H}^{t'}(K\cap V_i) +\mathcal{H}^{t'}(G^{\star}).
\end{split}
\end{equation*}
From (\ref{V}) and $t'>t$, we get $\mathcal{H}^{t'}(K\cap V_i)=0.$ Therefore $$\mathcal{H}^{t'}(G^{\star})=\mathcal{H}^{t'}(K)>0.$$
In particular, $G^{\star}\ne\varnothing$. Since $\{V_i\}$ is a basis of $X$, for any open set $V\subset X$ with $V\cap G^{\star}
\ne\varnothing$, we have $\dim_{\rm H}(G^{\star}\cap V)>t$. 

The statement $(2)$ can be obtained similarly with $(1)$ by applying Corollary 1 of Joyce and Preiss \cite{JP}. We omit the details.
\end{proof}

The following lemma is directly derived from the definitions of Hausdorff dimension and Hausdorff measure.
\begin{lemma}\label{hauslem2}
If $G\subset X$ with $\dim_{\rm H}(G)>t$, then there is $k_0\ge1$ such that for $k\ge k_0$, there are at least $b^{-kt}$ elements in 
$\mathcal{Q}_k$ intersecting G, that is 
\[\#\{Q\in\mathcal{Q}_k\colon Q\cap G\ne\varnothing\}\ge b^{-kt}.\]
\end{lemma}

For $k\ge1$, and $Q\in\mathcal{Q}_k$, denote by $B_Q(\in\mathcal{B}_k)$ the closed ball containing $Q$ given in Section 3.1. 
Write $X_Q$ for the indicator function of the event $\{\xi_n\in B_Q {\rm ~for ~some ~}n\in\mathfrak{I}'_k \}$. Then 
$$\mathbb{E}(X_Q)=\mathbb{P}\bigg(\bigcup_{n\in\mathfrak{I}'_k}\{\xi_n\in B_Q\}\bigg).$$ 

\begin{remark}\label{pro}
In the lemma below, we estimate ${\rm Cov}(X_{Q},X_{Q'})$ for some $Q',Q\in\mathcal{Q}_k$. Here we use $\mathfrak{I}'_k$ 
instead of $\mathfrak{I}_k$. Hence for any pair $n, m \in \mathfrak{I}'_k $, we have ${\rm dist}(n,m)\ge ck$ so that we can apply 
the exponential mixing property of $\{\xi_n\}$ to derive that 
$$\frac{\sum_{n\in \mathfrak{I}'_k}\sum_{m\in\mathfrak{I}'_k \atop n>m}\rho^{n-m}\mathbb{P}(\xi_n\in B_Q)}
{\big(\sum_{n\in \mathfrak{I}'_k} \mathbb{P}(\xi_n\in B_Q)\big)^2}\to 0,$$
as $k\to\infty$, which is crucial in our proofs. 
\end{remark}



\begin{lemma}\label{hauslem3}
Suppose $\a<s$ and $\epsilon>0$. 
There exists a constant $k_0\ge1$ such that for $k\ge k_0$, the following statements hold. \\
(1) If $B_Q\cap B_{Q'}=\varnothing$, where $Q,Q'\in \mathcal{Q}_k$, then 
\[{\rm Cov}(X_{Q},X_{Q'})<\epsilon\mathbb{E}(X_{Q})\mathbb{E}(X_{Q'}).\]
(2) There is a constant $0<M_0<\infty$ such that 
$$\max_{Q\in\mathcal{Q}_k}\#\{Q'\in\mathcal{Q}_k\colon B_Q\cap B_{Q'}\ne\varnothing \}\le M_0.$$
\end{lemma}

\begin{proof}
Let $Q,\,Q'\in\mathcal{Q}_k$ such that
$B_Q\cap B_{Q'}=\varnothing$. Then
\begin{equation}\label{cov1}
\begin{split}
& {\rm Cov}(X_Q,X_{Q'})=\mathbb{E}(X_QX_{Q'})-\mathbb{E}(X_Q)\mathbb{E}(X_{Q'})\\
 &= \mathbb{P}\bigg(\bigcup_{n\in\mathfrak{I}'_k}\bigcup_{m\in\mathfrak{I}'_k}\{\xi_n\in B_Q,\xi_m\in B_{Q'}\}\bigg)
 -\mathbb{P}\bigg(\bigcup_{n\in\mathfrak{I}'_k}\{\xi_n\in B_Q\}\bigg)
 \mathbb{P}\bigg(\bigcup_{m\in\mathfrak{I}'_k}\{\xi_m\in B_{Q'}\}\bigg).
\end{split}
\end{equation}
Observe that
\begin{equation}\label{cov2}
\begin{split}
&\mathbb{P}\bigg(\bigcup_{n\in\mathfrak{I}'_k}\bigcup_{m\in\mathfrak{I}'_k}\{\xi_n\in B_Q,\xi_m\in B_{Q'}\}\bigg)\\
&=\mathbb{P}\bigg(\bigcup_{n\in\mathfrak{I}'_k}\{\xi_n\in B_Q\cap B_{Q'}\}\cup 
\bigcup_{n\in\mathfrak{I}'_k}\bigcup_{m\in\mathfrak{I}'_k\atop m\ne n}\{\xi_n\in B_Q,\xi_m\in B_{Q'}\}\bigg)\\
&=\mathbb{P}\bigg(\bigcup_{n\in\mathfrak{I}'_k}\bigcup_{m\in\mathfrak{I}'_k\atop m\ne n}\{\xi_n\in B_Q,\xi_m\in B_{Q'}\}\bigg),
\end{split}
\end{equation}
where the last equality follows from the fact that $B_Q\cap B_{Q'}=\varnothing$.

 Since $\{\xi_j\}_{j\ge1}$ is an exponentially mixing stationary process, if $n> m$, we obtain
\[\mathbb{P}(\xi_n\in B_Q,\xi_m\in B_{Q'})\le \mathbb{P}(\xi_n\in B_Q)\mathbb{P}(\xi_m\in B_{Q'})
+C\rho^{n-m}\mathbb{P}(\xi_n\in B_Q).\]
Whence we get an upper bound for (\ref{cov2}):
\begin{equation}\label{1}
\begin{split}
&\mathbb{P}\bigg(\bigcup_{n\in\mathfrak{I}'_k}\bigcup_{m\in\mathfrak{I}'_k\atop m\ne n}\{\xi_n\in B_Q,\xi_m\in B_{Q'}\}\bigg)\\
&\le\bigg(\sum_{n\in\mathfrak{I}'_k}\mathbb{P}(\xi_n\in B_Q)\bigg)\bigg(\sum_{m\in\mathfrak{I}'_k}\mathbb{P}(\xi_m\in B_{Q'})\bigg)
+\frac{C\rho^{ck}}{1-\rho^{ck}}\, \sum_{n\in\mathfrak{I}'_k}\mathbb{P}(\xi_n\in B_Q),
\end{split}
\end{equation}
and a lower bound for 
\begin{equation}\label{3}
\begin{split}
&\mathbb{P}\bigg(\bigcup_{n\in\mathfrak{I}'_k}\{\xi_n\in B_Q\}\bigg)\ge \sum_{n\in\mathfrak{I}'_k}
\mathbb{P}(\xi_n\in B_Q)-\sum_{n\in\mathfrak{I}'_k}\sum_{n'\in\mathfrak{I}'_k\atop n'\ne n}\mathbb{P}(\xi_n\in B_Q,\xi_{n'}\in B_Q)\\
&\ge \bigg(\sum_{n\in\mathfrak{I}'_k}\mathbb{P}(\xi_n\in B_Q)\bigg)\bigg(1-\sum_{n'\in\mathfrak{I}'_k}\mathbb{P}(\xi_{n'}\in B_Q)-
\frac{C\rho^{ck}}{1-\rho^{ck}}\bigg).
\end{split}
\end{equation}
Since $\a<s$, we can pick $q\in(\a,s)$. It follows from equality (\ref{alimsup}) that there exists $k_1\ge1$ such that  
$n_k\le b^{-kq}$ for all $k\ge k_1$.
 Hence for $k\ge k_1$,
\[\sum_{n\in\mathfrak{I}'_k}\mathbb{P}(\xi_n\in B_Q)\le c_1{c'_2}^sb^{ks}m_k\le c_1{c'}_2^s\Big(b^{ks}+(ck)^{-1}b^{k(s-q)}\Big),\]
which yields $\sum_{n\in\mathfrak{I}'_k}\mathbb{P}(\xi_n\in B_Q)\to 0$ as $k\to\infty$. From (\ref{3}), there is a $k_2\ge 1$ such that for $k\ge k_2$, 
 \begin{equation}\label{addition}
 \mathbb{E}(X_Q)\ge M_1\sum_{n\in\mathfrak{I}'_k}\mathbb{P}(\xi_n\in B_Q),
 \end{equation}
 where $0<M_1<1$ is a constant. By combining (\ref{cov1})--(\ref{addition}), we see that 
 \begin{equation*}\label{covs}
 \begin{split}
 {\rm Cov}(X_Q,X_{Q'})&\le \bigg(\sum_{n\in\mathfrak{I}'_k}\mathbb{P}(\xi_n\in B_Q)\bigg)\bigg(\sum_{m\in\mathfrak{I}'_k}
 \mathbb{P}(\xi_m\in B_{Q'})\bigg)\bigg(\sum_{n\in\mathfrak{I}'_k \atop n\ne m}\mathbb{P}(\xi_n\in B_Q)\\
 &\qquad +\sum_{m\in\mathfrak{I}'_k \atop m\ne n}\mathbb{P}(\xi_m\in B_{Q'})+\frac{4C\rho^{c k}}{1-\rho^{ck}}\bigg)
 +\frac{2C\rho^{ck}}{1-\rho^{ck}}\bigg(\sum_{n\in\mathfrak{I}'_k}\mathbb{P}(\xi_n\in B_Q)\bigg)\\
 &\le \bigg(\sum_{n\in\mathfrak{I}'_k}\mathbb{P}(\xi_n\in B_Q)\bigg)\bigg(\sum_{m\in\mathfrak{I}'_k}\mathbb{P}(\xi_m\in B_{Q'})\bigg)
 \Bigg\{\sum_{n\in\mathfrak{I}'_k}\mathbb{P}(\xi_n\in B_Q) \\
 &\qquad +\sum_{m\in\mathfrak{I}'_k}\mathbb{P}(\xi_m\in B_{Q'})
  +\frac{2C\rho^{ck}}{1-\rho^{ck}}\bigg(2+\Big(\sum_{m\in\mathfrak{I}'_k}\mathbb{P}(\xi_m\in B_{Q'})\Big)^{-1}\bigg)\Bigg\}.
 \end{split}
 \end{equation*}
 The inequality above together with (\ref{addition}) implies that for $k\ge k_2$, we have
\begin{equation}\label{num1}
\begin{split}
 \frac{{\rm Cov}(X_Q,X_{Q'})}{\mathbb{E}(X_Q)\mathbb{E}(X_{Q'})}&\le
\frac{1}{M_1^2}\Bigg\{\sum_{n\in\mathfrak{I}'_k}\mathbb{P}(\xi_n\in B_Q)+\sum_{m\in\mathfrak{I}'_k}\mathbb{P}(\xi_m\in B_{Q'})\\
&\qquad \quad+\frac{2C\rho^{ck}}{1-\rho^{ck}}\bigg(2+\Big(\sum_{m\in\mathfrak{I}'_k}\mathbb{P}(\xi_m\in B_{Q'})\Big)^{-1}\bigg)\Bigg\}.
 \end{split}
\end{equation}
Since $c>\frac{s-\alpha}{\log_{b} \rho}$, we derive that the right-hand side of (\ref{num1}) tends to $0$ as $k\to\infty$.

Thereby for any $\epsilon>0$, there is a $k_0\ge1$ satisfying for all $k\ge k_0$, if $Q,Q'\in\mathcal{Q}_k$ 
with $B_Q\cap B_{Q'}=\varnothing$,
we always have 
$${\rm Cov}(X_Q,X_{Q'})<\epsilon \mathbb{E}(X_Q)\mathbb{E}(X_{Q'}).$$
Notice that if the distance ${\rm dist}(Q,Q')\ge 2(c'_2-c_2)b^k$ for $Q,Q'\in\mathcal{Q}_k$, we have $B_Q\cap B_{Q'}
=\varnothing$. Then for $Q\in\mathcal{Q}_k$,
\[\#\{Q'\colon B_Q\cap B_{Q'}\ne\varnothing\}\le \#\{Q'\colon {\rm dist}(Q,Q')\le2(c'_2-c_2)b^k\}.\]
From Lemma \ref{number}, there exists a constant $0<M_0<\infty$ independent of $k$ such that 
$$\max_{Q\in\mathcal{Q}_k}\#\{Q'\colon B_Q\cap B_{Q'}\ne\varnothing\}\le M_0$$
holds for $k\ge k_0$.
\end{proof}

Now we are ready to prove Theorem \ref{hithaus1} and Theorem \ref{hithaus2}.
\begin{proof}[Proof of Theorem~\ref{hithaus1}]
	
	Firstly we show that $\dim_{\rm P}(G)<s-\alpha$ implies $\mathbb{P}(E\cap G\ne\varnothing)=0$. By Tricot \cite{Tri}, 
	it suffices to show that whenever $\overline{\dim}_{\rm B}(G)<s-\alpha$, then $E\cap G=\varnothing$ a.s.
	
	We denote by $C_{\ell_n}= C_{\ell_n}(G)$ a collection of the smallest number of the closed balls with radius $\ell_n$ 
	that cover the set $G$. Let $N_{\ell_n}(G) = \#C_{\ell_n}$. 
	Fixing an arbitrary $\eta>0$ such that $\eta\in (\overline{\dim}_{\rm B}(G),s-\alpha)$, we have
	\[\limsup_{n\to\infty} \frac{ \log N_{\ell_n}(G)}{-\log (\ell_n)} \le \overline{\dim}_{\rm B}(G)<\eta,\]
	so there exists an integer $n_0 \in\mathbb{ N}$ such that
	\begin{equation}\label{nln}
		N_{\ell_n}(G)<\ell_n^{-\eta}
	\end{equation}
	for all $n\ge n_0$. For any ball $B$ in $X$ with radius $\ell_n$, since $\{\xi_n\}_{n\ge1}$ is a stationary process, we have
	\[\mathbb{P}\{I_n\cap B\ne \varnothing\}=\mathbb{P}(\xi_n\in B(x_B,2\ell_n))=\mu(B(x_B,2\ell_n))\le c_12^s\ell_n^s,\]
	where $I_n=B(\xi_n,\ell_n)$ and $x_B$ is the center of $B$. Note that the event
	\[\{I_n\cap G\ne\varnothing\}\subset \bigcup_{B\in C_{\ell_n}}\{I_n\cap B\ne\varnothing\}.\]
	We derive from this and (\ref{nln}) that
	\[\mathbb{P}\{I_n\cap G\ne\varnothing\}\le \sum_{B\in C_{\ell_n}}\mathbb{P}\{I_n\cap B\ne\varnothing\}
	\le N_{\ell_n}(G)c_12^s\ell_n^s<c_12^s\ell_n^{s-\eta}\]
	for all $n \ge n_0$. Hence the series $\sum_{n=1}^{\infty}\mathbb{P}\{I_n\cap G\ne\varnothing\}$ converges by the 
	definition of $\alpha$ and $\eta <s-\a$. By the Borel--Cantelli Lemma, we have 
	\[\mathbb{P}\{I_n\cap G\ne\varnothing~\, \text{ i.o.}\} =0.\]
	That is, $E\cap G=\varnothing$ a.s.

Now we prove that if $\dim_{\rm H}(G)>s-\alpha$, then $\mathbb{P}(E\cap G\ne\varnothing)=1$. By Lemma \ref{hauslem1}, 
we may assume that $G$ is compact and satisfies $\dim_{\rm  H}(G\cap V)>s-\alpha$ whenever $V\subset X$ is an open 
set with $G\cap V\ne\varnothing$. We choose constants $\beta$ and $t$ such that 
$\dim_{\rm  H}(G\cap V)>t>s-\beta>s-\alpha$.

Denote by $U(\xi_n,\ell_n)$ the open ball with center $\xi_n$ and radius $\ell_n$.
It suffices to show that 
\begin{equation}\label{finale}
\mathbb{P} \big\{U(\xi_n,\ell_n)\cap G\cap V\ne \varnothing~\, \hbox{ i.o.} \big\}=1.
\end{equation}
Letting $V$ run over a countable basis of $X$, we have $G\cap\bigcup_{n=k}^{\infty}B(\xi_n,\ell_n)$ is dense 
a.s. and relatively open in $G$ for any $k\ge1$. Then from Baire's category theorem we derive 
$G\cap\limsup\limits_{n\to\infty} U(\xi_n,\ell_n)\ne\varnothing$ a.s.,  which implies that 
$E\cap G\ne \varnothing$ a.s.

Now we prove that the equality (\ref{finale}) holds. By (\ref{alimsup}) we have 
\[\alpha=\limsup_{k\to\infty}\frac{\log_{b^{-1}}n_k}{k}>\beta,\] 
then there are infinitely many $k$ such that $n_k\ge b^{-k\beta}$. This implies that the set defined as
\[\mathcal{N}=\{k\ge1\colon n_k\ge b^{-k\beta}\}\]
satisfies $\#\mathcal{N}=\infty$. 

Fix an open set $V$ with $G\cap V\ne\varnothing$. Define 
\[\mathcal{Z}_k=\{Q\in\mathcal{Q}_k\colon Q\cap G\cap V\ne\varnothing\}.\]
From Lemma \ref{hauslem2}, we have $N_k=\#\mathcal{Z}_k\ge b^{-tk}$ for all $k$ large enough. 
For $Q\in\mathcal{Q}_k$, there exists $B(x_Q,c'_2b^k)\in\mathcal{B}_k$, denoted by $B_Q$. 
For $k\in\mathcal{N}$, define 
\[S_k=\#\{Q\in\mathcal{Z}_k\colon \xi_n\in B_Q {\rm ~for ~some ~}n\in\mathfrak{I}'_k \},\]
that is $S_k=\sum_{Q\in\mathcal{Z}_k} X_Q$.
For $\epsilon>0$ and $Q\in\mathcal{Z}_k$, write
\[
\mathcal{D}_k(Q)=\big\{Q'\in \mathcal{Z}_k\colon {\rm Cov}(X_{Q},X_{Q'})\ge \epsilon\mathbb{E}(X_{Q})\mathbb{E}(X_{Q'})\big\},\] 
then from Lemma \ref{hauslem3}, there is a constant $0<M'_0<\infty$ such that $\max_{Q\in \mathcal{Z}_k} 
\#\mathcal{D}_k(Q)\le M'_0$ for $k$ large enough.
Since $${\rm Var}(S_k)=\sum_{Q\in\mathcal{Z}_k}\sum_{Q'\in\mathcal{Z}_k}{\rm Cov}(X_Q,X_{Q'}),$$ 
and  
$$ {\rm Cov}(X_Q,X_{Q'})=\mathbb{E}(X_QX_{Q'})-\mathbb{E}(X_Q)\mathbb{E}(X_{Q'})\le\mathbb{E}(X_Q), $$
it follows that 
\begin{equation*}\label{cov3}
\begin{split}
{\rm Var}(S_k)&\le\sum_{Q\in\mathcal{Z}_k}\bigg(\sum_{Q'\in \mathcal{Z}_k \setminus\mathcal{D}_k(Q)}\epsilon\mathbb{E}(X_Q)
\mathbb{E}(X_{Q'})+M'_0\mathbb{E}(X_Q)\bigg)\\
&\le \epsilon \big(\sum_{Q\in\mathcal{Z}_k}\mathbb{E}(X_Q)\big)^2+M'_0\sum_{Q\in\mathcal{Z}_k}\mathbb{E}(X_Q).
\end{split}
\end{equation*}
From (\ref{3}), we have 
\begin{equation*}
\mathbb{E}(S_k)=\sum_{Q\in\mathcal{Z}_k}\mathbb{E}(X_Q)
\ge \sum_{Q\in\mathcal{Z}_k}\bigg(\sum_{n\in\mathfrak{I}'_k}\mathbb{P}(\xi_n\in B_Q)\bigg)\,
\bigg(1-\sum_{n\in\mathfrak{I}'_k}\mathbb{P}(\xi_n\in B_Q)-\frac{2C\rho^{ck}}{1-\rho^{ck}}\bigg). 
\end{equation*}
It yields that there is a constant $0<M_2<\infty$ such that for all $k\in\mathcal{N}$ large enough, we have
\[\mathbb{E}(S_k)\ge M_2N_km_kb^{ks}\ge M_2(ck)^{-1}b^{k(s-t-\beta)}.\]
Recall $t+\beta>s$, if $k\in \mathcal{N}$ and $k\to\infty$, we get $\mathbb{E}(S_k)\to \infty$.

Combining these and Chebyshev's inequality, we obtain 
\begin{equation*}
\mathbb{P}(S_k=0)\le \frac{{\rm Var}(S_k)}{\mathbb{E}^2(S_k)}\le \epsilon+\frac{M'_0}{\mathbb{E}(S_k)}.
\end{equation*} 
Therefore $$\limsup_{k\in \mathcal{N}\atop k\to\infty}\mathbb{P}(S_k=0)=0.$$
We observe that 
\begin{equation*}
\{U(\xi_n,\ell_n)\cap G\cap V\ne \varnothing~ \, \hbox{ i.o.}\}\supset \{S_k>0~\, \hbox{ i.o.}\}.
\end{equation*}
Finally this together with Fatou's lemma implies that
\[\mathbb{P}(U(\xi_n,\ell_n)\cap G\cap V\ne \varnothing~ \, \hbox{ i.o.})\ge \limsup_{ k\to\infty}\mathbb{P}(S_k>0)=1.\]
This finishes the proof of Theorem \ref{hithaus1}.
\end{proof}

For proving Theorem \ref{hithaus2}, we will make use of the following lemma, which is an analogue of Lemma 3.4 in \cite{KPX}, 
where $X=[0,1]^N$ was considered. Here $(X,d)$ is an Ahlfors regular metric space.
\begin{lem}\label{lemlow}
Equip $X$ with the Borel $\sigma$-algebra. Suppose that $A=A(\omega)$ is a random set in $X$ (i.e., the indicator function 
$\chi_{A(\omega)}(x) $ is jointly measurable) such that for any analytic set $G\subset X$ with $\dim_{\rm H}G>\gamma$, 
we have 
$$\mathbb{P}(A\cap G\ne\varnothing)=1.$$
Then if $\dim_{\rm H}G>\gamma$, we have
\[\dim_{\rm H}(A\cap G)\ge \dim_{\rm H}G-\gamma,~~~~{\rm a.s.}\]
\end{lem}

We postpone the proof of Lemma \ref{lemlow} to the end of this section. Let us first prove Theorem~\ref{hithaus2}. 
\begin{proof}[Proof of Theorem~\ref{hithaus2}]
Firstly we prove that $\dim_{\rm H}(E\cap G)\le \dim_{\rm P}G+\alpha-s$ a.s. when $\dim_{\rm P}G+\alpha-s\ge0$, and  it suffices to show that 
\[\dim_{\rm H}(E\cap G)\le \overline{\dim}_{\rm B}G+\alpha-s\quad{\rm a.s.}\]
Let $\mathcal{C}_{\ell_n}$ be a collection of the closed balls with radius $\ell_n$ whose union covers $G$ such that 
$\mathcal{N}_{\ell_n}(G)=\# \mathcal{C}_{\ell_n}$ is the smallest. Let $\xi:=\overline{\dim}_{\rm B}G+\a-s$, then for 
any $\epsilon>0$, we have 
\[\xi+s-\a+\epsilon>\overline{\dim}_{\rm B}G\ge \limsup_{n\to\infty}\frac{\log \mathcal{N}_{\ell_n}(G)}{-\log (\ell_n)},\]
which implies that  
\begin{equation}\label{nlnub}
\mathcal{N}_{\ell_n}(G)<\ell_n^{\a-s-\xi-\epsilon}
\end{equation}
for all $n$ large enough. For any $\delta>0$, we choose $N\ge1$ such that for all $n\ge N$, we have $\ell_n<\delta$ and 
inequality (\ref{nlnub}) holds.

Let
\[\mathcal{A}_n=\{B\in \mathcal{C}_{\ell_n}: B\cap I_n\ne \varnothing\}.\]
Notice that
\[\mathbb{E}(\#\mathcal{A}_n)\le \sum_{B\in\mathcal{A}_n}\mathbb{P}(I_n\cap B\ne\varnothing)\le c_12^s\ell_n^{\a-\xi-\epsilon},\]
and 
\[E\cap G\subset \bigcup_{n=N}^{\infty}I_n\cap G\subset  \bigcup_{n=N}^{\infty}\bigcup_{B\in\mathcal{A}_n}B.\]
Hence for any $\theta>\xi$, by choosing $\epsilon>0$ with $2\epsilon<\theta-\xi$, we have
\[\mathbb{E}\big(\mathcal{H}_{\delta}^{\theta}(E\cap G)\big)\le  \sum_{n=n_1}^{\infty}\mathbb{E}(\#\mathcal{A}_n)(2\ell_n)^{\theta}
\le\sum_{n=n_1}^{\infty}c_12^{\theta+s}\ell_n^{\a+s}.\]
By the definition of $\a$, we obtain that
\[\mathbb{E}\big(\mathcal{H}_{\delta}^{\theta}(E\cap G)\big)<\infty.\]
Therefore $\mathcal{H}^{\theta}(E\cap G)<\infty$ a.s. which implies 
$\dim_{\rm H}(E\cap G)\le\theta$ a.s. Hence $\dim_{\rm H}(G\cap E)\le \dim_{\rm P}G+\alpha-s$ a.s. 

When $\dim_{\rm P}G+\alpha-s<0$, by Theorem \ref{hithaus1}, $\mathbb{P}(E\cap G=\varnothing)=1$. Hence 
$\dim_{\rm H}(E\cap G)=-\infty$ a.s.

The lower bound in Theorem \ref{hithaus2} follows from Lemma \ref{lemlow}. This finishes the 
proof of Theorem \ref{hithaus2}.
\end{proof}

It remains to prove Lemma \ref{lemlow}. The fractal percolation ${\Gamma_t}$ used in Lemma \ref{lemlow} 
can be constructed as follows.

Let $0<p<1$. For each $Q\in\mathcal{Q}$, let $Z(Q)$ be a random variable taking value 1 with probability $p$ and 
value 0 with probability $1-p$. We assume that these random variables are independent for different $Q \in\mathcal{ Q}$. 
We define the random fractal percolation 
set as
\[\Gamma(p)=  \bigcap_{n\in\mathbb{N}}\bigcup_{Q\in\mathcal{Q}_n \atop Z (Q)=1}\overline{Q},\]
where $\overline{Q}$ is the closure of the set $Q$. When $p=2^{-t}<1$, for convenience, we denote $\Gamma_t
=\Gamma(2^{-t})$.

Notice that there is a nature tree structure behind the definition of $\mathcal{Q}$, which we describe now.   

Label each $Q\in\mathcal{Q}$ with a vertex $v_Q$ and let $T$ be a graph with vertex set $\{v_Q\}_{Q\in\mathcal{Q}}$. 
There exists an edge between  vertices $v_{Q_{n,i}}$ and $v_{Q_{m,j}}$ if and only if $|n-m|=1$ and $Q_{n,i}\cap 
Q_{m,j}\ne\varnothing$. Then $T$ is a tree with $v_X$ as its root. The boundary of the tree $\partial T$ consists of 
all infinite paths $(v_{0,i_0}v_{1,i_1}v_{2,i_2}\dots)$, where ${v_{m,i_m}}=v_{Q_{m,i_m}}$ for $Q_{m,i_m} \in
\mathcal{Q}_m$ and $Q_{m,i_m} \subset Q_{n,i_n}$, if $m\ge n$. We call these infinite 
paths $rays$. Then we can define a projection $\Pi\colon\partial T \to X$ as
\begin{equation} 
\Pi\colon (v_{{0,i_0}}v_{{1,i_1}}v_{{2,i_2}}\dots) \mapsto \bigcap_{n=0}^{\infty}\overline{Q_{n,i_n}}.
\end{equation}
 Note that $\Pi(\partial T) = X$.  
 For $v=(v_{{0,i_0}}v_{{1,i_1}}v_{{2,i_2}}\dots)$, $u =(u_{{0,i_0}}u_{{1,i_1}}u_{{2,i_2}}\dots)\in\partial T$, we define
\[
	\kappa(v,u)=
	\begin{cases}
		0& \ \ \text { if $u=v$,}\\
		b^{\min\{j\colon v_{{j,i_j}}\ne u_{{j,i_j}}\}}  &\  \ \text{ if $u\ne v$.}
	\end{cases} 
\]
Then $(\partial T,\kappa)$ is a metric space. We claim that for every $G\subset X$, we have
\begin{equation}\label{dimeq}
\dim_{\rm H}G= \dim_{\rm H}^{\kappa}(\Pi^{-1}G),
\end{equation}
where $\dim_{\rm H}^{\kappa}$ is the Hausdorff dimension in the metric space $(\partial T,\kappa)$. 

Now we prove this claim.  For any $x,y\in X$, let $n$ be the maximal integer with $x,y\in Q_n\in\mathcal{Q}_n$. By 
Theorem \ref{thm:nest}, we have $\diam(Q_n) \le 2c_2'b^{n}$. Then we get 
\[d(x, y)\le \diam(Q_n)\le 2c_2'b^n= 2c_2'b^{-1}\kappa(\Pi^{-1}x,\Pi^{-1}y),\]
which derives $\dim_{\rm H}G\le \dim_{\rm H}^{\kappa}(\Pi^{-1}G)$.

Further, let $t>\dim_{\rm H}G$. For $\epsilon>0$, there exists a $\delta$-covering $\{U_i\}_{i\ge1}$ of $G$ which satisfies 
$\sum_{i\ge1} (\diam{U_i})^t<\epsilon$. 
For $i\ge1$, let $n(i)$ be the integer with 
\begin{equation}\label{diam}
2c'_2b^{n(i)}\le \diam(U_i)<2c'_2b^{n(i)-1}.
\end{equation}
Write
\[\mathcal{A}(U_i)=\{Q\in \mathcal{Q}_{n(i)}\colon Q\cap U_i\ne\varnothing\}.\]
Then using (\ref{diam}), there is a constant $0<M_3<\infty$ such that for all $i\ge1$, $\#\mathcal{A}(U_i)\le M_3$. Note that
\[\Pi^{-1}G\subset \Pi^{-1}\Big(\bigcup_{i\ge1} U_i\Big)\subset \Pi^{-1}\Big(\bigcup_{i\ge1}\bigcup_{Q\in\mathcal{A}(U_i)}Q\Big).\]
Hence 
\begin{equation}\label{haus}
\mathcal{H}_{\delta}^t(\Pi^{-1}G)\le\sum_{i\ge1}\sum_{Q\in\mathcal{A}(U_i)}\kappa\Big((\Pi^{-1}(Q)\Big)^t\le \frac{M_3b^t}{(2c'_2)^t}
\sum_{i\ge1} (\diam U_i)^t<\epsilon.
\end{equation}
The second inequality in (\ref{haus}) holds due to (\ref{diam}).
Then we have $\mathcal{H}^t(\Pi^{-1}G)=0$ which derives that $t\ge\dim_{\rm H}^{\kappa}(\Pi^{-1}G)$. Therefore 
$\dim_{\rm H}G\ge\dim_{\rm H}^{\kappa}(\Pi^{-1}G)$. This proves (\ref{dimeq}).

For $0<p<1$, let $T$ be the tree defined above. $Percolation$ at level $p$ on $T$ is obtained by removing each edge of 
$T$ with probability $1-p$ and retaining it with probability $p$, with mutual independence among edges. The random 
graph connecting the root which is left will be denoted by $\widetilde{\Gamma}(p)$. Then the law of ${\Gamma}(p)$ given 
by $\mathbb{P}$ is the same as that of $\Pi(\widetilde{\Gamma}(p)$). By combining this with (\ref{dimeq}), we obtain the 
following analogue of the result of \cite[p.~957]{lyons} in our metric space setting. See also Lemma 5.1 in \cite{Pere} for 
the case of $X=[0,1]^N$. 

\begin{lemma}[\cite{lyons}]\label{perco}
Let $p=2^{-t}<1$. For any analytic set $G\subset X$, the following statements hold.
\begin{enumerate}
\item If $\dim_{\rm H}(G)<t$, then $\Gamma_t\cap G=\varnothing$ almost surely.
\item If $\dim_{\rm H}(G)>t$, then $\Gamma_t\cap G\ne\varnothing$ with positive probability.
\item If $\dim_{\rm H}(G)>t$, then $\|\dim_{\rm H}(G\cap \Gamma_t)\|_{\infty}=\dim_{\rm H}G-t$, where 
the $L^{\infty}$ norm is the essential supremum in the underlying probability space.
 \end{enumerate}
\end{lemma}

We end this section with a proof of Lemma \ref{lemlow}, by extending the method in \cite{KPX} to the Ahlfors regular 
metric spaces.

\begin{proof}[Proof of Lemma \ref{lemlow}]
For $t < \dim_{\rm H} (G)-\gamma$, let $\Gamma_t$ be a fractal percolation at level $2^{-t}$ in $X$ (\cite{lyons,Pere}), 
which is independent of $A$, and $\dim_{\rm H}(\Gamma_t)=s-t$ a.s. By Lemma \ref{perco}, we have 
$\mathbb{P}(\Gamma_t \cap G\ne\varnothing)>0$ whenever $\dim_{\rm H}(G) > t$, whereas $\mathbb{P}(\Gamma_t 
\cap G\ne\varnothing) =0$ with $\dim_{\rm H}(G) <t$, and
\[ \|\dim_{\rm H}(G\cap \Gamma_t)\|_{\infty} =\dim_{\rm H}(G)-t.\]
Let $\widehat{ \Gamma}_t$ be a union of countably many independent and identically distributed copies of $\Gamma_t$. 
The Borel-Cantelli lemma implies that
   \begin{equation}\label{hit}
   \mathbb{P}(\widehat{\Gamma}_t \cap G\ne\varnothing)=
	\begin{cases}
	0&  \text{if $\dim_{\rm H}(G)<t$,}\\
	1&  \text{if $\dim_{\rm H}(G)>t$.}
	\end{cases}
 \end{equation}
Also we have
\[\dim_{\rm H}(\widehat{\Gamma}_t\cap G)=\dim_{\rm H}(G)-t>\gamma,\quad{\rm a.s.}\]
By the condition given in the lemma, we have $A\cap G\cap\widehat{\Gamma_t}\ne\varnothing$ a.s. in the product space.
Using (\ref{hit}), we get $\dim_{\rm H}(A\cap G)\ge t$ a.s. Letting $t$ tend to $\dim_{\rm H}G-\gamma$ along rational numbers, 
we complete our proof.
\end{proof}

\section{Limsup Random Fractals in Metric Space}

In order to prove Theorem \ref{rc1}, we first extend the results on hitting probabilities of limsup random fractals in 
\cite{KPX} to metric spaces.

Let $\mathcal{Q}=\{\mathcal{Q}_n\}_{n\ge0}$ be the collection of generalized dyadic cubes in $(X,d)$ given in 
Section 3.1. To make sure the boundaries of sets in $\mathcal{Q}_n$ are covered, we will use $2Q$, $Q\in\mathcal{Q}_n$ instead of generalized dyadic cubes in this section. However, we still denote them by $Q$ and $\mathcal{Q}_n$ respectively for simplicity of notation. For each $n\ge1$, let $\{Z_n(Q),Q\in \mathcal{Q}_n\}$ be a collection of random variables, each 
taking values in $\{0,1\}$. 

Let
\[A(n)=\bigcup_{Q\in \mathcal{Q}_n,\atop Z_n(Q)=1}{Q^o},
\]
where $Q^o$ is the interior of $Q$. The random set 
\[A=\limsup_{n\to\infty}A(n)\]
is called a limsup random fractal associated to $\{Z_n(Q),n\ge 1,Q\in \mathcal{Q}_n\}$. 

We assume the following conditions (H1)--(H2), where (H1) is more general than 
Condition 4 in \cite{KPX}. We allow the probability $P_n(Q)$ to depend not only on the level $n$, but also on 
the cubes $Q\in\mathcal{Q}_n$.

(H1)\, Suppose that for every $n\ge1$, and $Q\in \mathcal{Q}_n$ the probability $P_n(Q):=
\mathbb{P}(Z_n(Q)=1)$ satisfies
\[\lim_{n\to\infty}\frac{\min_{Q\in\mathcal{Q}_n}\log_{b^{-1}}P_n(Q) }{n}=-\gamma_1,\]
and
\[\lim_{n\to\infty}\frac{\max_{Q\in\mathcal{Q}_n}\log_{b^{-1}}P_n(Q) }{n}=-\gamma_2,\]
where $\gamma_1,\gamma_2>0$ are constants.

\begin{remark}\label{subsequence}
From the proofs of Theorems \ref{thm2}, \ref{thm2'} and Corollary \ref{limsupest}, we see that they still hold if 
the condition (H$1$) is replaced by the weaker condition (H$1'$):

(H$1'$)\, For some constants $\gamma_1,\gamma_2>0$,
\[\limsup_{n\to\infty}\frac{\min_{Q\in\mathcal{Q}_n}\log_{b^{-1}}P_n(Q) }{n}=-\gamma_1,\]
\[\limsup_{n\to\infty}\frac{\max_{Q\in\mathcal{Q}_n}\log_{b^{-1}}P_n(Q) }{n}=-\gamma_2,\]
and there exists an increasing sequence of positive integers $\{n_i\}$ with $\lim_{i\to\infty}\frac{n_{i+1}}{n_i}=1$ 
such that 
\[
\lim_{i\to\infty}\frac{\min_{Q\in\mathcal{Q}_{n_i}}\log_{b^{-1}}P_{n_i}(Q) }{n_i}=-\gamma_1,\]
and
\[\lim_{i\to\infty}\frac{\max_{Q\in\mathcal{Q}_{n_i}}\log_{b^{-1}}P_{n_i}(Q) }{n_i}=-\gamma_2.\]
\end{remark}

The next condition is concerned with the strength of dependence among the random variables $\{Z_n(Q),\, 
n\ge 1,\, Q\in \mathcal{Q}_n\}$.  It is a slight modification of Condition 5 in \cite{KPX}.

(H$2$)\, A bound on correlation length: for any $\epsilon>0$, define 
\[f(n,\epsilon)=\max_{Q\in \mathcal{Q}_n}\#\{Q'\in \mathcal{Q}_n\colon{\rm Cov}(Z_n(Q),Z_n(Q'))\ge\epsilon P_n(Q)P_n(Q')\}.\]
Suppose that there is a constant $\delta\ge0$ such that for all $\epsilon>0$, 
\begin{equation}\label{ast}
	\limsup_{n\to\infty}\frac{1}{n}\log_{b^{-1}}f(n,\epsilon)\le\delta. \tag{$\ast$}
\end{equation}



The following theorem characterizes the hitting probability of the limsup random set $A$. It extends 
Theorem 3.1 in \cite{KPX}.

\begin{theorem}\label{thm2}
Assume that $A=\limsup\limits_{n\to\infty}A(n)$ is a limsup random set that satisfies the conditions 
({H}1) and ({H}2). Then for any analytic set $G\subset X$,
	\[\mathbb{P}(A\cap G\ne\varnothing)=
	\begin{cases}
	0&  \text{if $\dim_{\rm P}(G)<\gamma_2$,}\\
	1&  \text{if $\dim_{\rm P}(G)>\gamma_1+\delta$.}
	\end{cases}\]
\end{theorem}

For proving Theorem \ref{thm2}, we will use the following lemma on upper box dimension. From \cite{MR}, 
for any $r>0$ and any bounded set $G\subset X$, let $N_{r}(G)$ be the smallest number of closed balls 
with radius $r$ covering $G$. Then
 \[
 \overline{\dim}_{\rm B}(G)=\limsup_{r\to0}\frac{\log N_{r}(G)}{-\log r}.
 \]
 
 From it, we have the following lemma whose proof is standard and thus omitted.
 
\begin{lemma}\label{prop:box}
Let $\{k_i\}_{i\ge1}$ be an increasing sequence of positive integers satisfying (\ref{ki}). Then for any bounded 
set $G\subset X$, 
\begin{equation}\label{box}
\overline{\dim}_{\rm B}(G)=\limsup_{i\to\infty}\frac{\log N_{a_1r^{k_i}}(G)}{-\log (a_1 r^{k_i})},
\end{equation}
where $a_1>0$ is a constant.
\end{lemma}


\begin{remark}\label{in}
For any bounded set $G\subset X$, consider $\mathcal{A}'=\{Q\in\mathcal{Q}_n\colon Q\cap G\ne\varnothing\}$, then 
$\#\mathcal{A}' \ge N_{c'_2b^n}(G)$ which derives from 
\[G\subset \bigcup_{Q\in\mathcal{A}'}Q\subset \bigcup_{x_Q\in Q\in\mathcal{A}'}B(x_Q,c'_2b^n),\]
where $x_Q\in Q$.
\end{remark}

\begin{proof}[Proof of Theorem~\ref{thm2}]
     The proof is a modification of that of Theorem 3.1 in \cite{KPX}. We include it for the sake of completeness. Firstly, 
     we show that $\dim_{\rm P}(G)<\gamma_2$ implies that $A\cap G=\varnothing$, a.s. It suffices to show that 
	whenever $\overline{\dim}_{\rm B}G<\gamma_2$, then $A\cap G=\varnothing$ a.s. Fix an arbitrary but small $\eta>0$ 
	so that $\overline{\dim}_{\rm B}(G)<\gamma_2-\eta$. 
	
	We denote by $\mathcal{C}_{b^n}= \mathcal{C}_{b^n}(G)$ a collection of the smallest number of the closed balls 
	with radius $b^n$ that cover the set $G$. Let $N_{b^n}(G) = \#\mathcal{C}_{b^n}$. For any 
	$\theta\in(\overline{\dim}_{\rm B}(G),\gamma_2-\eta)$, by Lemma \ref{prop:box},  we have
	\[\limsup_{n\to\infty} \frac{ \log N_{b^n}(G)}{-\log (b^n)} =\overline{\dim}_{\rm B}(G)<\theta,\]
	then there exists an integer $n(\theta)\ge1$ such that
	\begin{equation}\label{eq1}
		N_{b^n}(G)<b^{-n\theta}
	\end{equation}
	for all $n\ge n(\theta)$.  For any $W\in \mathcal{C}_{b^n}$, denote
	\[\mathcal{A}(W)=\{Q\in\mathcal{Q}_n\colon Q\cap W\ne\varnothing\}.\]
	Then we have
	\[G\subset \bigcup_{W\in \mathcal{C}_{b^n}}\bigcup_{Q\in\mathcal{A}(W)}Q.\]
	In the meanwhile, by Lemma \ref{number}, there is a constant $0<M<\infty$, which does not depend on $n$,  
	such that for all $W\in\mathcal{C}_{b^n}$, we have $\#\mathcal{A}(W)\le M$.	
	
	On the other hand, by condition (H1), for any $\eta>0$, there exists $n(\eta)$ such that for all $n\ge n(\eta)$,
	\begin{equation}\label{eq2}
		\max_{Q\in\mathcal{Q}_n}P_n(Q)\le b^{n(\gamma_2-\eta)}.
	\end{equation}
	It follows from (\ref{eq1}) and (\ref{eq2}) that for any $n\ge \max\{n(\theta),n(\eta)\}$,
	\begin{equation*}\label{one}
	\begin{split}
	\mathbb{P}\bigl(G\cap A(n)\ne\varnothing\bigr)&\le\mathbb{P}\bigg( \bigcup_{W\in \mathcal{C}_{b^n}}\bigcup_{Q\in\mathcal{A}(W)}
	Q\cap A(n)\ne\varnothing\bigg)\\
	&\le Mb^{-n\theta}\max_{Q\in \mathcal{Q}_n}\mathbb{P}\bigl(Q\cap A(n)\ne\varnothing\bigr)\\
	&= Mb^{-n\theta}\max_{Q\in \mathcal{Q}_n}P_n(Q)\le Mb^{n(\gamma_2-\eta-\theta)}.
	\end{split}
	\end{equation*}
	Since $\theta<\gamma_2-\eta$, then the series $\sum_{n\ge1}\mathbb{P}\bigl(G\cap A(n)\ne\varnothing\bigr)$ is convergent. 
	The Borel-Cantelli lemma implies that $G\cap A(n)=\varnothing$ a.s. for all $n$ large enough. This shows that $A\cap G=\varnothing$ a.s.

	In the following, we prove that if $\dim_{\rm P}G>\gamma_1+\delta$, then $\mathbb{P}(A\cap G\ne\varnothing)=1$.
	 From Lemma \ref{hauslem1} $(2)$, we can find a compact subset $G_{\star}\subset G$ such that for all 
	open sets $V$, whenever $G_{\star}\cap V \ne \varnothing$, then $\overline{\dim}_{\rm B}(G_{\star}\cap V)>\gamma_1+\delta$. 
	
	Fix an open set $V$ such that $V\cap G_{\star}\ne \varnothing$. Denote 
	\[\widetilde{\mathcal{A}}_n=\{Q\in \mathcal{Q}_n\colon Q^o\cap V\cap G_{\star}\ne\varnothing\}.\]
	Let $\mathcal{N}_n$ be the total number of $\widetilde{\mathcal{A}}_n$. Then
	\[G_{\star}\cap V\subset \bigcup_{Q\in\widetilde{\mathcal{A}}_n}Q\subset \bigcup_{Q\in\widetilde{\mathcal{A}}_n}B\big(x_Q,2c_2'b^n\big),
	\]
	where $x_Q\in Q$.
	Using Lemma \ref{prop:box}, we have
	\[ \limsup_{n\to\infty}\frac{\log N_{2c_2'b^n}(V\cap G_{\star}) }{-\log (2c_2'b^{n})}=\overline{\dim}_{\rm B}(V\cap G_{\star})>\gamma_1+\delta.\]
	For any $\eta\in(\gamma_1+\delta,\overline{\dim}_{\rm B}(V\cap G_{\star}))$, we derive that
	\[N_{2c_2'b^n}(V\cap G_{\star})\ge (2c'_2)^{-\eta}b^{-n\eta}\]
	holds for infinitely many $n$. 
	Hence $\mathcal{N}_n\ge (2c'_2)^{-\eta}b^{-n\eta}$ for infinitely many $n$. This implies the set 
	\begin{equation}\label{nn}
		\mathfrak{N}:=\{i\ge1:\mathcal{N}_{n_i}\ge  (2c'_2)^{-\eta}b^{-{n_i}\eta}\}
	\end{equation}
	satisfies $\#\mathfrak{N}=\infty$.We define 
	\[S_i:=\sum_{Q\in \widetilde{\mathcal{A}}_{n_i}}Z_{n_i}(Q),\]
	where $S_i$ is the total number of sets $Q\in \mathcal{Q}_{n_i}$ with $Q\cap V\cap G_{\star}\cap A({n_i})\ne\varnothing$. Observe that 
	\[\{A(n)\cap G_{\star}\cap V\ne\varnothing ~ \, \hbox{i.o.}\}\supset \{S_i>0~\, \hbox{ i.o.}\}.\]
	We need only show that $\mathbb{P}(S_i>0~\, \hbox{ i.o.})=1$. Firstly, we estimate
	\[{\rm Var}(S_i)=\sum_{ Q\in\widetilde{\mathcal{A}}_{n_i}}\sum_{Q'\in \widetilde{\mathcal{A}}_{n_i}}{\rm Cov}(Z_{n_i}(Q),Z_{n_i}(Q')).\]
	Fix $\epsilon>0$ and for each $Q\in \mathcal{Q}_{n_i}$, let $\mathcal{G}_{n_i}(Q)$ denote the collection of all $Q'\in \mathcal{Q}_{n_i}$ such that 
	\begin{enumerate}
		\item $Q'\cap V\cap G_{\star}\ne \varnothing$;
		\item ${\rm Cov}(Z_{n_i}(Q),Z_{n_i}(Q'))\le \epsilon P_{n_i}(Q)P_{n_i}(Q')$.
	\end{enumerate}
	That is 
	\[\mathcal{G}_{n_i}(Q)=\{Q'\in \widetilde{\mathcal{A}}_{n_i}\colon {\rm Cov}(Z_{n_i}(Q),Z_{n_i}(Q'))\le \epsilon P_{n_i}(Q)P_{n_i}(Q')\}.\]
	If $Q'\in \mathcal{Q}_{n_i}$ satisfies (1) but not (2), then we say $\mathcal{B}_{n_i}(Q)$,
	\[\mathcal{B}_{n_i}(Q)=\{Q'\in\widetilde{\mathcal{A}}_{n_i}\colon {\rm Cov}(Z_{n_i}(Q),Z_{n_i}(Q'))> \epsilon P_{n_i}(Q)P_{n_i}(Q')\}.\]
	 Then 
	\begin{equation}\label{Var}
		\begin{split}
		{\rm Var}(S_i)&\le \epsilon \bigg(\sum_{Q\in\widetilde{\mathcal{A}}_{n_i}}P_{n_i}(Q)\bigg)^2
		+ \bigg(\max_{Q\in \mathcal{Q}_{n_i}}\#\mathcal{B}_{n_i}(Q)\bigg)\bigg(\sum_{Q\in\widetilde{\mathcal{A}}_{n_i}}P_{n_i}(Q)\bigg).\\	
		\end{split}
	\end{equation}
	The last term follows from the fact that ${\rm Cov}(Z_n(Q),Z_n(Q'))\le \mathbb{E}(Z_n(Q))=P_n(Q)$. Recalling the notation 
	of (H2), we have
	\[{\rm Var}(S_i)\le  \epsilon \bigg(\sum_{Q\in\widetilde{\mathcal{A}}_{n_i}}P_{n_i}(Q)\bigg)^2+f({n_i},\epsilon)
	\bigg(\sum_{Q\in\widetilde{\mathcal{A}}_{n_i}}P_{n_i}(Q)\bigg).
	\]
	Combining this with the Paley-Zygmund inequality, we obtain 
	\begin{equation}\label{sn}
		\begin{split}
		\mathbb{P}(S_i>0)&\ge \frac{\bigl(\mathbb{E}(S_i)\bigr)^2}{\mathbb{E}(S_i^2)}=\frac{1}{1+\frac{{\rm Var}(S_i)}{(\mathbb{E}(S_i))62}}\\
		&\ge \frac{1}{1+\epsilon+\frac{f({n_i},\epsilon)}{\sum_{Q\in\widetilde{\mathcal{A}}_{n_i}}P_{n_i}(Q)}},
		\end{split}
	\end{equation}
	since $\mathbb{E}(S_i)=\sum_{Q\in\widetilde{\mathcal{A}}_{n_i}}P_{n_i}(Q)$. By the conditions (H1) and (H2), for any $\theta>0$ 
	with $2\theta<\eta-\delta-\gamma_1$, there exists $N$ such that for all $n\ge N$, we have
	\[f({n},\epsilon)\le b^{-(\delta+\theta)n}~ {\rm and}~\min_{Q\in\mathcal{Q}_n}P_{n}(Q)\ge b^{(\gamma_1+\theta)n}.\]
	Thus from (\ref{nn}) and arbitrariness of $\theta$, we have
	\begin{equation}\label{si0}
	\limsup_{ i\in \mathfrak{N}\atop i\to\infty}\frac{f({n_i},\epsilon)}{\sum_{Q\in\widetilde{\mathcal{A}}_{n_i}}P_{n_i}(Q)}
	\le \limsup_{ i\in \mathfrak{N}\atop i\to\infty}(2c'_2)^{-\eta}b^{-n_i(2\theta+\delta+\gamma_1-\eta)}=0.
	\end{equation}
	By inequalities (\ref{sn}), (\ref{si0}), and Fatou's lemma, we derive that
	\[\mathbb{P}(S_i>0~\, \hbox{ i.o.} )\ge \limsup_{ i\in \mathfrak{N}\atop i\to\infty}\mathbb{P}(S_i>0)=1.\]
	
	Define the open set $B(n):=\bigcup_{k=n}^{\infty}A(k)$ for $n\ge1$. It follows that
	\[\mathbb{P}(B(n)\cap G_{\star}\cap V\ne \varnothing,~\forall n\ge1)=1\]
	for every open set $V$ with $G_{\star}\cap V\ne\varnothing$. Since compact metric spaces are separated, then there 
	exists a countable basis for open sets of $(X,d)$. Letting $V$ run over the countable basis, we obtain that for all $n\ge1$, 
	the set $B(n)\cap G_{\star}$ is a.s. dense in $G_{\star}$. Since $G_{\star}$ is a complete metric space, by Baire's 
	category theorem, we have $\cap_{n=1}^{\infty}B(n)\cap G_{\star}$ is a.s. dense in $G_{\star}$.  In particular, 
	$A\cap G_{\star}\ne\varnothing$ with probability one.
\end{proof}

For $n\ge1$, let $\mathcal{B}_n=\{B(x_{n,i},2c'_2b^n)\colon i\in\mathbb{N}_n\}$ where $x_{n,i}$, $c'_2$ are given in Section 3.1, and $\mathcal{B}
=\{\mathcal{B}_n\}_{n\ge0}$. For each $n\ge1$, let $\{Z_n(B),B\in \mathcal{B}_n\}$ be a collection of random variables, 
each taking values in $\{0,1\}$. Let
\[F(n)=\bigcup_{B\in \mathcal{B}_n,\atop Z_n(B)=1} B^o.\]
Define the random set 
\[F=\limsup_{n\to\infty}F(n).\]

\begin{theorem}\label{thm2'}
Assume that $F$ satisfies the corresponding conditions (H1) and (H2). Then for any analytic set $G\subset X$, if 
$\dim_{\rm P}(G)>\gamma_1+\delta$, we have
	\[\mathbb{P}(F\cap G\ne\varnothing)=1.\]
\end{theorem}

\begin{proof}
By replacing the generalized dyadic cubes $\mathcal{Q}$ by balls in $\mathcal{B}$, we obtain the theorem directly from 
the proof of Theorem \ref{thm2}.
\end{proof}

\begin{corollary}\label{limsupest}
Suppose $A$ is a discrete limsup random fractal satisfying conditions (H1) and (H2) with $\delta=0$. Then for any 
analytic set $G\subset X$, with probability one,
\begin{align*}
            \dim_{\rm H}(A\cap G) \left\{\begin{array}{ll}
            \le \dim_{\rm P}(G)-\gamma_2& \quad \text{if~}\dim_{\rm P}(G)\ge \gamma_2,\\[2ex]
             =-\infty & \quad \text{if~}\dim_{\rm P}(G)< \gamma_2,\\[2ex]
             \ge\dim_{\rm H}(G)-\gamma_1 &\quad \text{if~}\dim_{\rm H}(G)>\gamma_1.
          \end{array}\right.
           \end{align*}
In particular, if $\gamma_1=\gamma_2<s$, then $\dim_{\rm H}(A)=s-\gamma_1$, {\rm a.s.}
\end{corollary}
\begin{proof}
It suffices to prove $\dim_{\rm H}(A\cap G)\le \overline{\dim}_{\rm B}G-\gamma_2$ a.s., if $\dim_{\rm P}(G)\ge \gamma_2$. Let $\mathcal{C}_{b^n}$ and 
$\mathcal{A}(W)$, $W\in\mathcal{C}_{b^n}$ be the same as described in the proof of Theorem \ref{thm2}. Define 
\[S_n=\sum_{W\in \mathcal{C}_{b^n}}\sum_{Q\in\mathcal{A}(W)}Z_n(Q).\]
Let $\xi=\overline{\dim}_{\rm B}G-\gamma_2$, then for $n$ large enough, we have
\[\mathbb{E}(S_n)=\sum_{W\in \mathcal{C}_{b^n}}\sum_{Q\in\mathcal{A}(W)}P_n(Q)\le Mb^{-n(\xi+2\epsilon)},\]
 where $0<M<\infty$ is a constant independent of $n$. Hence from the arbitrariness of $\epsilon$, it follows that 
 for any $\theta>\xi$, $\mathbb{E}(\sum_nS_nb^{\theta n})<\infty$.

Hence
\[\mathcal{H}^{\theta}(A\cap G)\le \sum_{n=m}^{\infty}S_nb^{n\theta}<\infty \quad{\rm a.s.}\]
which derives that $\dim_{\rm H}(A\cap G)\le \dim_{\rm P}G-\gamma_2$ a.s.

Finally, the lower bound of $\dim_{\rm H}(A\cap G)$ follows from Lemma \ref{lemlow}. The proof is complete.
\end{proof}

\section{Proof of Theorem ~\ref{rc1} }

\begin{proof}[Proof of Theorem~\ref{rc1}]
The proof is divided into three parts. In order to show that $\dim_{\rm P}(G)>s-\alpha$ implies 
$\mathbb{P}(E\cap G\ne\varnothing)=1$, we will use the hitting probability of limsup random fractals 
in Section 4. This is done in Parts (i) and (ii). Part (iii) determines the packing dimension of  $E\cap G$.

(i) Construction of a limsup random fractal $E_{\star}\subset E$. 

Firstly, we recall some notations from Sections 3, 4. For any $k \ge 2$, $\mathcal{B}_k=\{B(x_{k,i},2c'_2b^k):i\in\mathbb{N}_k\subset
\mathbb{N}\}$, $\mathfrak{I}_k = \{n\ge1: \ell_n\in [b^{k-1},b^{k-2})\}$, $n_k =\#\mathfrak{I}_k$ and $\mathfrak{I}'_k$ 
is a maximal collection of points in $\mathfrak{I}_k$ having mutual distances at least $ck$,
where $c$ is a given constant with $c>\frac{\alpha-s}{\log_{b^{-1}} \rho}$, and $\rho$ appears in Definition \ref{exp}. 
In this way, any pair of integers $n, m \in \mathfrak{I}'_k $ are at least of distance $c k$ from each other. Also, recall that 
 $m_k=\#\mathfrak{I}'_k =\lceil (c k)^{-1} n_k\rceil.$

For every $J\in \mathcal{B}_k$, define 
\begin{align*}
Z_k(J) = \left\{\begin{array}{cl}
1& {\rm if} ~ \exists ~n\in\mathfrak{I}'_k ~{\rm such ~that}~J\subset I_n=B(\xi_n,\ell_n),\\[2ex]
0& {\rm otherwise}.
\end{array}\right.
\end{align*}

Let $A(k)$ be the union of the interiors of sets in $\mathcal{B}_k$ that are contained in some $I_n$ in $\mathfrak{I}'_k$ 
with length $\ell_n \in [b^{k-1}, b^{k-2})$, that is 
\[A(k)=\bigcup_{J\in\mathcal{B}_k \atop Z_k(J)=1} J^o.\]
We observe that
$$A(k)\subset \bigcup_{n\in\mathfrak{I}'_k} I_n.$$
Define $E_{\star} := \limsup\limits_{k\to\infty} A(k)$. From the above, we have $E_{\star} \subset E$.

(ii) Hitting probability of $E_{\star}$. 

Now let $G \subset X$ be an analytic set such that $\dim_{\rm P} (G) >s-\alpha$. We show $\mathbb{P}( E_{\star} \cap G\ne \varnothing) = 1$. 

For every $J \in \mathcal{B}_k$, the probability
\[\mathbb{P}(Z_k(J)=1) =\mathbb{P}\{\exists~ n\in\mathfrak{I}'_k~ {\rm such~that}~J\subset I_n\}.\]
Denote the above probability by $P_k(J)$.

Write $J=B(x_J,2c'_2b^k)$. By using the stationarity on $\{\xi_n\}$, we derive that, 
\begin{equation}\label{upper}
\begin{split}
P_k(J)&\le \sum_{n\in\mathfrak{I}'_k}\mathbb{P}(J\subset I_n)
=\sum_{n\in\mathfrak{I}'_k}\mathbb{P}(\xi_n\in B(x_J,\ell_n-2c'_2b^k))\\
&\le c_1\sum_{n\in\mathfrak{I}'_k} (\ell_n-2c'_2b^k)^s \le c_3 m_k b^{ks},
\end{split}
\end{equation}
where $c_3=c_1(\frac{1}{b^2}-2c'_2)^s$ is a constant.

On the other hand, 
\begin{equation}\label{lower}
\begin{split}
P_k(J)&\ge \sum_{n\in\mathfrak{I}'_k}\mathbb{P}(J\subset I_n)-\sum_{n\in\mathfrak{I}'_k}\sum_{m\in\mathfrak{I}'_k \atop m\ne n}
\mathbb{P}(J\subset I_n, J\subset I_m).
\end{split}
\end{equation}

Since $\{\xi_n\}_{n\ge1}$ is stationary and exponentially mixing, if $n>m$, we have 
\[\mathbb{P}(J\subset I_n,J\subset I_m)\le\mathbb{P}(J\subset I_m)\mathbb{P}(J\subset I_n)+C\rho^{n-m}\mathbb{P}(J\subset I_n),\]
where $C, \rho$ are constants. We notice that $n,m\in\mathfrak{I}'_k$ derives $n-m\ge c k$, then 
\begin{equation}\label{later}
\begin{split}
&\sum_{n\in\mathfrak{I}'_k}\sum_{m\in\mathfrak{I}'_k \atop m\ne n}\mathbb{P}(J\subset I_n, J\subset I_m)\\
&\le \bigg(\sum_{n\in\mathfrak{I}'_k}\mathbb{P}(J\subset I_n)\bigg)\bigg(\sum_{m\in\mathfrak{I}'_k \atop m\ne n}
\mathbb{P}(J\subset I_m)\bigg)+2C\sum_{n\in\mathfrak{I}'_k}\sum_{m\in\mathfrak{I}'_k \atop m<n}\rho^{n-m}
\mathbb{P}(J\subset I_n)\\
&\le \bigg(\sum_{n\in\mathfrak{I}'_k}\mathbb{P}(J\subset I_n)\bigg)\bigg(\sum_{m\in\mathfrak{I}'_k \atop m\ne n}
\mathbb{P}(J\subset I_m)\bigg)+\frac{2C\rho^{c k}}{1-\rho^{ck}}\bigg(\sum_{n\in\mathfrak{I}'_k}\mathbb{P}(J\subset I_n)\bigg).
\end{split}
\end{equation}
It follows from (\ref{lower}) and (\ref{later}) that  
\begin{equation}\label{lowerr}
\begin{split}
P_k(J)&\ge \bigg(\sum_{n\in\mathfrak{I}'_k}\mathbb{P}(J\subset I_n)\bigg)\bigg(1-\sum_{m\in\mathfrak{I}'_k \atop m\ne n}
\mathbb{P}(J\subset I_m)-\frac{2C\rho^{c k}}{1-\rho^{ck}}\bigg)\\
&\ge c_4m_k b^{ks}\Big(1-c_3m_kb^{ks}-\frac{2C\rho^{c k}}{1-\rho^{ck}}\Big),
\end{split}
\end{equation}
where $c_4=c_1^{-1}(\frac{1}{b}-2c'_2)^s$ is a constant. Combining (\ref{upper}) and (\ref{lowerr}), together with the condition (C), 
we derive that
\begin{equation}\label{pka}
\limsup_{k\to\infty} \frac{\max_{J\in\mathcal{Q}_k}\log_{b^{-1}} P_k(J)}{k} =\limsup_{k\to\infty} \frac{\min_{J\in\mathcal{Q}_k}\log_{b^{-1}} 
P_k(J)}{k}= \alpha-s,
\end{equation}
and there is an increasing sequence of integers $\{k_i\}$ that satisfies (\ref{ki}) such that 
\begin{equation}\label{pki}
\lim_{i\to\infty} \frac{\max_{J\in\mathcal{Q}_{k_i}}\log_{b^{-1}} P_{k_i}(J)}{k_i} =\lim_{i\to\infty} \frac{\min_{J\in\mathcal{Q}_{k_i}}
\log_{b^{-1}} P_{k_i}(J)}{k_i}= \alpha-s.
\end{equation}

Next we verify that there is a bound on correlation length. 

First we estimate ${\rm Cov}(Z_k(J_1)Z_k(J_2))$, where $J_1,J_2\in \mathcal{B}_k $ with $d(J_1,J_2)\ge b^{k-3}$. From (\ref{later}), 
we have
\begin{equation}\label{e1}
\begin{split}
\mathbb{E}(Z_k(J_1)Z_k(J_2))
&=\mathbb{P}(Z_k(J_1)=1,Z_k(J_2)=1)\\
&=\mathbb{P}\{\exists~ m,n\in\mathfrak{I}'_k~{\rm such ~that~}J_1\subset I_n~{\rm and}~J_2\subset I_m\}\\
&\le \sum_{n\in\mathfrak{I}'_k}\sum_{m\in\mathfrak{I}'_k \atop m\ne n}\mathbb{P}(J_1\subset I_n,J_2\subset I_m)\\
&\le  \sum_{n\in\mathfrak{I}'_k}\mathbb{P}(J_1\subset I_n)\sum_{m\in\mathfrak{I}'_k \atop m\ne n}\mathbb{P}(J_2\subset I_m)
+\frac{2C\rho^{c k}}{1-\rho^{ck}}\sum_{n\in\mathfrak{I}'_k}\mathbb{P}(J_1\subset I_n).
\end{split}
\end{equation}

Since 
\begin{equation}\label{cov}
\begin{split}
{\rm Cov}(Z_k(J_1),Z_k(J_2))&=\mathbb{E}(Z_k(J_1)Z_k(J_2))-\mathbb{E}(Z_k(J_1))\mathbb{E}(Z_k(J_2))\\
&=\mathbb{E}(Z_k(J_1)Z_k(J_2))-P_k(J_1)P_k(J_2),
\end{split}
\end{equation}
from the first inequality in (\ref{lowerr}) and (\ref{e1}), we get
\begin{equation}\label{final}
\begin{split}
&{\rm Cov}(Z_k(J_1),Z_k(J_2))\\
&\le \bigg(\sum_{n\in\mathfrak{I}'_k}\mathbb{P}(J_1\subset I_n)\bigg)\bigg(\sum_{m\in\mathfrak{I}'_k}\mathbb{P}(J_2\subset I_m)\bigg)
\bigg(\sum_{n\in\mathfrak{I}'_k \atop n\ne m}\mathbb{P}(J_1\subset I_n)+\sum_{m\in\mathfrak{I}'_k \atop m\ne n}\mathbb{P}(J_2\subset I_m)
+\frac{4C\rho^{c k}}{1-\rho^{ck}}\bigg)\\
&\quad +\frac{2C\rho^{ck}}{1-\rho^{ck}}\bigg(\sum_{n\in\mathfrak{I}'_k}\mathbb{P}(J_1\subset I_n)\bigg)\\
&\le  \bigg(\sum_{n\in\mathfrak{I}'_k}\mathbb{P}(J_1\subset I_n)\bigg)\bigg(\sum_{m\in\mathfrak{I}'_k}\mathbb{P}(J_2\subset I_m)\bigg)
\bigg(\sum_{n\in\mathfrak{I}'_k \atop n\ne m}\mathbb{P}(J_1\subset I_n)+\sum_{m\in\mathfrak{I}'_k \atop m\ne n}\mathbb{P}(J_2\subset I_m)
+\frac{4C\rho^{ck}}{1-\rho^{ck}}\\
&\quad +\frac{2C\rho^{ck}}{(1-\rho^{ck})\sum_{m\in\mathfrak{I}'_k}\mathbb{P}(J_2\subset I_m)}\bigg).
\end{split}
\end{equation}

We notice that from (\ref{lowerr}) there exists a constant $0<M_4<\infty$ such that  
\[\mathbb{E}(Z_k(J))=P_k(J)\ge M_4\sum_{n\in\mathfrak{I}'_k}\mathbb{P}(J\subset I_n)\]
holds for all $J\in\mathcal{B}_k$ and all large enough $k$. Then the inequality (\ref{final}) reads as follows
\begin{equation*}
\begin{split}
{\rm Cov}(Z_k(J_1),Z_k(J_2))&\le \frac{1}{M_4^2}\mathbb{E}(Z_k(J_1))\mathbb{E}(Z_k(J_2))\bigg(2c_3m_kb^{ks}
+\frac{4C\rho^{c k}}{1-\rho^{ck}}+\frac{2C\rho^{c k}}{c_4(1-\rho^{ck})m_kb^{ks}}\bigg).
\end{split}
\end{equation*}
Here choosing ${k_i}$ in (\ref{pki}), then $\lim\limits_{i\to\infty}m_{k_i}b^{k_is}=0$, and
\[\lim_{i\to\infty}\frac{\rho^{c k_i}}{m_{k_i}b^{k_is}}\le\lim_{i\to\infty}\frac{c k_i\rho^{ck_i}}{n_{k_i}b^{sk_i}}
=\lim_{i\to\infty}\frac{c k_i\rho^{c k_i}}{b^{(s-\alpha)k_i}}=0,\]
due to $\rho^{c}< b^{s-\alpha}$. We derive from the equalities above that for any $\epsilon>0$,
\[{\rm Cov}\Big(Z_{k_i}(J_1),Z_{k_i}(J_2)\Big)<\epsilon \mathbb{E}(Z_{k_i}(J_1))\mathbb{E}(Z_{k_i}(J_2))\]
holds for $i$ large enough. From Lemma \ref{number}, this implies that there is a constant $0<M_5<\infty$ independent 
of $k_i$ such that $f(k_i,\epsilon)\le M_5$, and recall
\[f(k,\epsilon)=\max_{J_2\in\mathcal{B}_k}\#\Big\{J_1\in\mathcal{B}_k:{\rm Cov}\Big(Z_k(J_1),Z_k(J_2)\Big)\ge\epsilon 
\mathbb{E}(Z_k(J_1))\mathbb{E}(Z_k(J_2))\Big\}.\]
In particular,
\[\lim_{i\to\infty}\frac{\log_{b^{-1}} f(k_i,\epsilon)}{k_i}=0.\]
Thus we have shown that the condition (H2) in Section 4 is satisfied with $\delta=0$.

Now that we have verified that $E_{\star}$ satisfies the conditions (H$1'$) and (H2) in Section 4, we apply Theorem 
\ref{thm2'} to conclude $E_{\star}\cap G\ne\varnothing$ a.s., which yields $\mathbb{P}(E\cap G\ne\varnothing)=1$. 

(iii) The packing dimension of $E\cap G$.

On the same probability space, let  $E'$ be a random covering set that is independent of $E$ and is associated with 
$\{\xi'_n\}$ and $\{\ell'_n\}$, where $\{\xi'_n\}$ is an exponentially mixing stationary process, and $\{\ell'_n\}$ satisfies 
the condition (C) with the Besicovitch-Taylor index $\alpha'<s$. 

Let the compact set $G_{\star}$ and the open sets $A(k)$ be as described in the proofs of Theorems \ref{thm2} 
and \ref{rc1}, respectively. Let $\{A'(k) \}$ be the sequence of open sets corresponding to $E'$. If $\dim_{\rm P}G 
>s-\min\{\alpha,\alpha'\}$, by the first part of Theorem \ref{rc1}, we have
\[\mathbb{P}\bigg(\Big(\bigcup_{k=n}^{\infty}A(k)\Big)\cap V \cap G_{\star}\ne\varnothing,\forall n\ge1\bigg)
=\mathbb{P}\bigg(\Big(\bigcup_{k=n}^{\infty}A'(k)\Big)\cap V \cap G_{\star}\ne\varnothing,\forall n\ge1\bigg)=1\]
for all open set $V$ satisfying $V\cap G_{\star}\ne\varnothing$. By independence, we have
\[
\mathbb{P}\bigg(\Big(\bigcup_{k=n}^{\infty}A(k)\Big)\cap V \cap G_{\star}\ne\varnothing, ~
\Big(\bigcup_{k=n}^{\infty}A'(k)\Big)\cap V 
\cap G_{\star}\ne\varnothing,~\forall n\ge1\bigg)=1.
\]
Let $V$ run over a countable basis of $X$, then $\big\{\bigcup_{k=n}^{\infty}A(k)\cap G_{\star}\big\}_{n\ge1}
\cup\big\{\bigcup_{k=n}^{\infty}A'(k)\cap 
G_{\star}\big\}_{n\ge1}$ is a countable collection of open, dense subsets of the complete metric space $G_{\star}$. 
Baire's category theorem implies that 
\[\mathbb{P}(E\cap E'\cap G_{\star}~{\rm is ~dense~in} ~G_{\star})=1.\]
In particular, $E\cap E'\cap G_{\star}\ne\varnothing$ a.s. Hence $\mathbb{P}(E\cap E'\cap G\ne\varnothing)=1$.

Now we consider the random covering set $E'$, and regard $E\cap G$ as the target set. By Theorem \ref{hithaus1}, 
we must have  
$\dim_{\rm P}(E\cap G)\ge s-\alpha'$ a.s. Hence we have proved that $\dim_{\rm P}(G)>s-\min\{\alpha,\alpha'\}$ implies 
$\dim_{\rm P}(E\cap G)\ge s-\alpha'$ a.s. Consequently, if $\dim_{\rm P}(G)>s-\alpha$, then for $\alpha'\in(s-\dim_{\rm P}G,\alpha)$ 
we have $\dim_{\rm P}(E\cap G)\ge s-\alpha'$ a.s. Letting $\alpha'$ tend to $s-\dim_{\rm P}G$ along rational numbers, we get 
\[\dim_{\rm P}(G\cap E)\ge \dim_{\rm P}G,\quad {\rm a.s.}\]
Hence $\dim_{\rm P}(G\cap E)=\dim_{\rm P}G$ a.s. This finishes the proof of Theorem \ref{rc1}.
\end{proof}

\section{Applications}

In this section we present some dynamical systems which satisfy all conditions of our results.
\subsection{Continued fraction dynamical system}

Let $T_G$ be the Gauss map on $(0, 1]$. The Gauss measure $\mu$ on $(0, 1]$ are given by $$d\mu=\frac{1}{\log 2}\frac{dx}{(1+x)}.$$ 
From \cite{CRN}, $T_G$ preserves the Gauss measure $\mu$ which is equivalent to the Lebesgue measure $\mathcal L$. By Philipp \cite{Phi}, 
the system $((0, 1], T_G)$ is exponentially mixing with respect to the Gauss measure $\mu$. Hence our results are applicable to the continued 
fraction dynamical system. For a given point $x\in (0, 1]$, let
\[E_{G}(x)=\{y\in (0, 1]\colon y\in B(T_{G}^nx,\ell_n)~{\rm i.o.}\}.\]

\begin{theorem}\label{cfds1}
Let $\mu$ be the Gauss measure. For any analytic set $G\subset (0, 1]$ we have, for $\mu$-{\rm a.e.} $x\in(0, 1]$
 \begin{equation*}
\begin{array}{cl}
E_G(x)\cap G=\varnothing& {\rm if}~\dim_{\rm P}(G)<1-\alpha,\\[2ex]
E_G(x)\cap G\ne\varnothing& {\rm if}~\dim_{\rm H}(G)>1-\alpha.
 \end{array}
\end{equation*}
 Furthermore under the condition (C), if $\dim_{\rm P}(G)>1-\alpha$,
\begin{equation*}
	 E_{G}(x)\cap G\ne\varnothing\quad{\rm a.e.}
\end{equation*}
\end{theorem}

\begin{theorem}\label{cfds2}
For any analytic set $G\subset (0, 1]$ we have a.e.
\begin{align*}
            \dim_{\rm H}(E_G(x)\cap G) \left\{\begin{array}{ll}
            \le \dim_{\rm P}(G)+\alpha-1& \quad \text{if~}\dim_{\rm P}(G)\ge 1-\alpha,\\[2ex]
             =-\infty & \quad \text{if~}\dim_{\rm P}(G)< 1-\alpha,\\[2ex]
             \ge\dim_{\rm H}(G)+\alpha-1 &\quad \text{if~}\dim_{\rm H}(G)>1-\alpha.
          \end{array}\right.
           \end{align*}
Moreover, if $\dim_{\rm P}(G)>1-\alpha$ and the condition (C) is satisfied, then $\dim_{\rm P}(E_G(x)\cap G)=\dim_{\rm P}(G)$~~{\rm a.e.}
\end{theorem}

\subsection{The $\beta$-dynamical system}
 For a real number $\beta>1$, define the transformation $T_\beta:[0,1]\to[0,1]$ by $$T_\beta: x\mapsto \beta x\bmod 1.$$
Let $\mu$ be the Parry measure with the density $$
h(x)={\left(\int_{0}^1 \sum_{n: T^n1<x}\frac{1}{\beta^n}dx\right)^{-1} \sum_{n: T^n1<x}}\frac{1}{\beta^n}.
$$
It was shown by \cite{renyi} that the Parry measure $\mu$ is invariant under $T_\beta$ and equivalent to the Lebesgue measure 
$\LL$. Hence $\mu$ is Ahlfors regular. From Philipp \cite{Phi}, $T_{\beta}$ is exponentially mixing. Combining these, we obtain 
that the $\beta$-dynamical system $([0,1], T_\beta)$ satisfies all the conditions stated in our results. For a given $x\in[0,1]$, 
define the dynamical covering set
\[E_{\beta}(x)=\{y\in [0,1]\colon y\in B(T_{\beta}^nx,\ell_n)~{\rm i.o.}\}.\]

\begin{theorem}\label{beta1}
Let $([0,1], T_\beta)$ be the $\beta$-dynamical system endowed with the Parry measure $\mu$. 
 Then for any analytic set $G\subset [0,1]$ we have
 \begin{equation*}
\begin{array}{cl}
E_{\beta}(x)\cap G=\varnothing& {\rm if}~\dim_{\rm P}(G)<1-\alpha,\\[2ex]
E_{\beta}(x)\cap G\ne\varnothing& {\rm if}~\dim_{\rm H}(G)>1-\alpha
 \end{array}
\end{equation*}
for $\mu$-{\rm a.e.} $x\in[0,1]$. Under the corresponding condition (C), if $\dim_{\rm P}(G)>1-\alpha$,
\begin{equation*}
	 E_{\beta}(x)\cap G\ne\varnothing\quad {\rm a.e.}
\end{equation*}
\end{theorem}

\begin{theorem}\label{beta2}
Let $([0,1], T_\beta)$ be the $\beta$-dynamical system endowed with the Parry measure $\mu$. For any analytic set $G\subset [0,1]$ we have a.e.
\begin{align*}
            \dim_{\rm H}(E_{\beta}(x)\cap G) \left\{\begin{array}{ll}
            \le \dim_{\rm P}(G)+\alpha-1& \quad \text{if~}\dim_{\rm P}(G)\ge 1-\alpha,\\[2ex]
             =-\infty & \quad \text{if~}\dim_{\rm P}(G)< 1-\alpha,\\[2ex]
             \ge\dim_{\rm H}(G)+\alpha-1 &\quad \text{if~}\dim_{\rm H}(G)>1-\alpha.
          \end{array}\right.
           \end{align*}
Moreover, if $\dim_{\rm P}(G)>1-\alpha$ and  the condition (C) is satisfied, then $\dim_{\rm P}(E_{\beta}(x)\cap G)=\dim_{\rm P}(G)$ ~{\rm a.e.}
\end{theorem}

\subsection{The middle-third Cantor set}
Our results are applicable to the middle-third Cantor set $C_{1/3}$. In fact, the results also hold for a range of homogeneous self-similar 
sets satisfying the open set condition.

 Let $T_3x=3x$ (mod)1 be the natural map on $C_{1/3}$, and $\mu$ be the standard Cantor measure. Let $\gamma=\log_3 2$ be 
 the Hausdorff dimension of $C_{1/3}$. From Lemma 3.2 in \cite{Wang}, $\mu$ is exponentially mixing. Also $\mu$ is Ahlfors $\gamma$-regular. Then all the conditions are fulfilled for Theorem \ref{main1} and \ref{main2}. Define the 
 dynamical covering set
\[E_3(x)=\{y\in C_{1/3}\colon y\in B(T_3^nx,\ell_n)~ \ \, {\rm i.o.}\},\]
where $x\in C_{1/3}$ is a given point.

\begin{theorem}\label{similar1}
Let $\mu$ be the standard Cantor measure. Then for any analytic set $G\subset C_{1/3}$, we have
\begin{equation*}
\begin{array}{cl}
E_3(x)\cap G=\varnothing& {\rm if}~\dim_{\rm P}(G)<\gamma-\alpha,\\[2ex]
E_3(x)\cap G\ne\varnothing& {\rm if}~\dim_{\rm H}(G)>\gamma-\alpha
 \end{array}
\end{equation*}
for $\mu$-{\rm a.e.} $x\in C_{1/3}$. Under the corresponding condition (C), if $\dim_{\rm P}(G)>\gamma-\alpha$, 
	$E_3(x)\cap G\ne\varnothing$ holds for $\mu$-{\rm a.e.} $x\in C_{1/3}$.
\end{theorem}

\begin{theorem}\label{similar2}
For any analytic set $G\subset C_{1/3}$ we have a.e.
\begin{align*}
            \dim_{\rm H}(E_3(x)\cap G) \left\{\begin{array}{ll}
            \le \dim_{\rm P}(G)+\alpha-\gamma& \quad \text{if~}\dim_{\rm P}(G)\ge \gamma-\alpha,\\[2ex]
             =-\infty & \quad \text{if~}\dim_{\rm P}(G)< \gamma-\alpha,\\[2ex]
             \ge\dim_{\rm H}(G)+\alpha-\gamma &\quad \text{if~}\dim_{\rm H}(G)>\gamma-\alpha.
          \end{array}\right.
           \end{align*}
Moreover, if $\dim_{\rm P}(G)>\gamma-\alpha$ and the condition (C) is satisfied, then $\dim_{\rm P}(E_3(x)\cap G)=\dim_{\rm P}(G)$~~{\rm a.e.}
\end{theorem}

\subsection*{Acknowledgements}
This research of Bing Li and Zhangnan Hu was supported by  NSFC 11671151 and Guangdong Natural Science Foundation 2018B0303110005. The research of Yimin Xiao is partially supported by the NSF grant DMS-1855185.

\end{document}